\documentclass[10pt,letterpaper]{amsart}
\usepackage{amssymb}
\usepackage{amsthm}
\usepackage{epsfig} 
\usepackage{epic,eepic}
\usepackage{xcolor}
\usepackage{comment}
\usepackage{enumitem}

\newtheorem{thm}{Theorem}[section]
\newtheorem{prop}[thm]{Proposition}

\newtheorem{cor}[thm]{Corollary}

\theoremstyle{definition}

\theoremstyle{remark}
\newtheorem{remark}[thm]{Remark}

\usepackage{hyperref}

\usepackage{esint}

\newcommand{\D}{\mathcal{D}}

\renewcommand{\S}{\mathcal{S}}
\newcommand{\R}{\mathbb{R}}
\newcommand{\p}{\partial}
\newcommand{\spt}{\operatorname{spt}}
\newcommand{\BUC}{\operatorname{BUC}}

\newcommand{\eps}{\epsilon}

\renewcommand{\tt}[1]{\text{#1}}

\numberwithin{equation}{section}

\begin{document}

\title{Stability of Vortex Patches in Channels}

 \author{Zelin Dong}
 \address{Department of Mathematics, The Chinese University Hong Kong, 
 Shatin, NT, Hong Kong SAR}
 \email{1155173731@link.cuhk.edu.hk}
% \urladdr{www.math.sc.edu/$\sim$howard} % Delete if not wanted.
 \author{Chenyun Luo}
 \address{Department of Mathematics, The Chinese University Hong Kong, 
 Shatin, NT, Hong Kong SAR}
 \email{cluo@math.cuhk.edu.hk}
\flushbottom
\maketitle
\thispagestyle{empty}

\begin{abstract}
In this paper, we investigate the orbital stability of vortex patches for the two-dimensional incompressible Euler equations in both a class of domains that satisfy the “weak finite volume condition” and a strip of arbitrary width. We establish that for suitable parameters \((\mu,\lambda)\), the penalized kinetic energy functional admits a minimizer, and that every such minimizer satisfies the elliptic equation \(\omega = \lambda(\psi - W x_2 - \gamma)_+\). Furthermore, we demonstrate that the set of minimizers is orbitally stable under the Eulerian dynamics. This work extends the variational framework developed by Abe and Choi \cite{abe2022stability} to domains that lack both spatial scaling invariance and horizontal translation invariance. The absence of these properties introduces substantial difficulties in the proof, as classical rearrangement and scaling arguments are no longer applicable. We overcome these obstacles by comparing the Green's function with that of the half-plane and exploiting the decay condition to formulate a concentration-compactness argument that ultimately yields the desired stability result.
\end{abstract}

\section{Introduction}

The two-dimensional Euler equation for the motion of an inviscid, incompressible fluid on a fixed domain $\Omega\subseteq \R^2_+$ is given in vorticity form by
\begin{align}
    \p_t \omega +u\cdot\nabla \omega&= 0, ~ \quad \omega(x,0)=\omega_0(x) \quad \tt{ in } \Omega.  \label{eqn:vorticity equation}
\end{align}
Here, $\omega$ is the vorticity of the flow, and the fluid velocity $u$ satisfies the slip boundary condition: 
\begin{align*}
    u\cdot \mathcal{N}=0,
\end{align*}
where \(\mathcal{N}\) is the unit outward normal. When the appropriate Biot-Savart law is applicable, $u$ is determined from $\omega$ in the sense that $$u(x,t)=\nabla^\perp \int_D G(x,y)\omega(y,t)dy,$$ and $G$ is the Green's function for the Dirichlet problem in $\Omega$. It will be convenient for us to define the stream function associated with $\omega$:
\begin{align}
    \psi(x,t)=\int_{\Omega} G(x,y) \omega(y,t)dy, \label{def:stream function}
\end{align}
so that $u(x,t)=\nabla^\perp \psi$. In addition, during this paper, when conducting a discussion at a fixed time, the time $t$ will be omitted, i.e., $\omega (\cdot, t) = \omega (\cdot), \psi (\cdot, t) = \psi (\cdot) $ and $u(\cdot, t) = u(\cdot) $.

% The two-dimensional incompressible fixed boundary Euler equations describe the motion of a fluid in two dimensions with a fixed boundary separating the fixed fluid region $D$ and the vacuum region. In the fluid region, the fluid velocity $u(t, x)$ and the pressure $p(t, x)$ satisfy the incompressible Euler equations with slip boundary condition:
% \begin{align}
%     \left\{
%     \begin{aligned}
%         \p_t u +u\cdot\nabla u+\nabla p &= 0, &&  \tt{ in } D, \\
%         \nabla\cdot u &= 0, && \tt{ in } D, \\
%         u \cdot \mathcal{N} &= 0, && \tt{ on } \partial D.
%     \end{aligned}\label{eqn:velocity equation}
%     \right.
% \end{align}
% where \(\mathcal{N}\) is the unit outward normal. When the Biot-Savart law hold, by applying \(\nabla^\perp\) to the velocity equation \eqref{eqn:velocity equation} and using the divergence-free condition, the vorticity \(\omega = \nabla^\perp \cdot u\) satisfies the transport equation
% \begin{align}
%     \left\{
%     \begin{aligned}
%         \p_t \omega +u\cdot\nabla \omega&= 0, &&  \tt{ in } D, \\
%         u&=\nabla^\perp(-\Delta)\omega, && \tt{ in } D, \\
%         u \cdot \mathcal{N} &= 0, && \tt{ on } \partial D.
%     \end{aligned}\label{eqn:vorticity equation}
%     \right.
% \end{align}

A natural question is whether solutions remain stable over long time intervals. In general, the answer is negative. For 2D Euler equations in a disk, Nadirashvili, Alexander, and Šverák \cite{kiselev2014small} constructed solutions whose vorticity gradient grows double‑exponentially. Similar mechanisms were exploited on the torus $\mathbb{T}^2$ by Zlato{\v{s}}, Andrej \cite{zlatovs2015exponential}, applied to smooth domains with an axis of symmetry by Xu \cite{xu2016fast}, and also for free-boundary problems by Hu, Luo, and Yao \cite{hu2024small}, demonstrating that instability can occur generically.

Nevertheless, stability can be established for certain domains with suitable classes of initial data. Cao and Wang \cite{cao2021nonlinear} proved nonlinear stability of highly concentrated vortex patches near non‑degenerate minima of the Robin function; Cao, Wan, and Wang \cite{cao2019nonlinear} showed orbital stability for kinetic‑energy‑maximising patches; Choi, Jeong, and Lim \cite{choi2022stability} studied the stability of monotone, nonnegative, and compactly supported vorticities in the half cylinder. Earlier work by Burton \cite{burton2005global} and Turkington \cite{turkington1987evolution} established foundational stability results for vortex patches using rearrangement techniques. However, all these results highly rely on the specific geometry of the domain, the assumptions about the initial data, and the techniques used (such as rearrangement or translation invariance).

In recent years, a variational approach based on minimizing a penalized energy has been successfully applied to prove the orbital stability of vortex patches. Abe and Choi \cite{abe2022stability} established the stability of Lamb dipoles in the half‑plane $\R_+^2$ by considering the minimization problem 
\begin{align}
    I_{\mu,\nu,\lambda}=\inf_{\omega \in K_{\mu,\nu}}\{-E_{2,\lambda}\}, \label{def:minimizing problem}
\end{align}
where $K_{\mu,\nu}$ is a class of admissible vorticities with fixed first moment and bounded $L^1$ norm, and $E_{2,\lambda}$ is the penalized kinetic energy. Their proof relies on two key properties of the half‑plane: i) the translation invariance property, which allows the use of rearrangement techniques, and ii) the spatial scaling invariance, which is used to obtain monotonicity and negativity and strictly decreasing in $\mu$ of  $I_{\mu,\nu,\lambda}$ to control minimizing sequences.

The variational approach of Abe and Choi \cite{abe2022stability} has recently been extended to more complex symmetric configurations. Choi, Jeong, and Yao \cite{choi2024stability} studied the orbital stability of a pair of opposite-signed Lamb dipoles under odd-odd symmetry, as well as concentrated vortices in a quadrant, using the monotonicity of the first moment and pointwise kernel estimates. Abe, Choi, Jeong, Sim, and Woo \cite{abe2025existence} demonstrate the existence and stability of Sadovskii vortices. However, those results are heavily reliant on the convenience of the (half) plane, i.e., the explicit structure of the Green's function, the scaling property, and rearrangement techniques. In contrast, domains meet the “weak finite volume condition” and infinite strip considered here lack these structures, requiring new tools to establish weak continuity of the kinetic energy and consequently the existence and stability of minimizers.

The main goal of this paper is to extend the variational approach to two classes of domains: 
\begin{enumerate} [label=\arabic*.]
    \item \label{def:case Omega=D} $\Omega=\D = D_u\cup D_l$, 
    where
        \begin{align}
            D_u := \mathcal{D} \cap \{x_2\ge 1\}, && D_l := \D \cap \{x_2<1\}, \label{def:upper and lower part of domain}
        \end{align}
        that {satisfies the following two conditions:
        \begin{itemize}
            \item The “weak finite volume condition”, i.e.,         \begin{align}
            \int_{D_l} x_2^p dx<\infty, \tt{ for some } p\geq 0,\tt{ and } ~ \tt{Vol}(D_u)
            % =\int_{D_u} x_1^0 dx
            <\infty.\label{def:weak_finite_volume_domain_assumption}
        \end{align}
        \item The Biot-Savart law is applicable on $\D$. 
        \end{itemize}}
        Without loss of generality, it suffices to study the case when \(p\ge 2\). This class includes a wide range of domains, including but not limited to all simply connected domains with (piecewise) smooth boundary \cite{lacave2019euler}. The weak tangency condition on non-smooth domains is discussed in \cite{gerard2013two} and \cite{galdi2011introduction}. In particular, this class covers all bounded domains, {domains with periodic boundaries}, and domains whose boundaries are graphs of some appropriate functions. These are of considerable importance in the study of free-boundary problems.
        
    \item \label{def:case Omega=S} $\Omega=\S$, where $\S$ is the infinite strip
    \begin{align}
        \S=\R \times (0,L), ~ L>0, \label{def:strip_domain_assumption}
    \end{align}
    whose Green’s function is given explicitly by
    \begin{align}
        G(x,y)=&\frac{1}{4\pi}\ln\Big(1+\frac{\sin(\frac{\pi}{L}x_2)\sin(\frac{\pi}{L}y_2)}{\sinh(\frac{\pi}{2L}(x_1-y_1))^2+\sin(\frac{\pi}{2L}(x_2-y_2))^2}\Big). \label{def:green's function for strip}
    \end{align}
    Such strips can be regarded as a generalization of the half-plane model to domains of finite width, or as a limit $p\to \infty$ of Class \ref{def:case Omega=D}. Nevertheless, due to the different widths $L$ and the absence of decay of the domain itself prevent a direct application of the methods from \cite{abe2022stability} or from Class \ref{def:case Omega=D}.
    
\end{enumerate}

% In class \ref{def:case Omega=D}, without loss of generality, it suffices to study the case \(p\ge 2\). In addition to this, the Biot–Savart law applies to many domains, including all simply connected domains with (piecewise) smooth boundary \cite{lacave2019euler}, while the weak tangency condition on non-smooth domains is discussed in \cite{gerard2013two} and \cite{galdi2011introduction}. Consequently, our results encompass a substantial class of domains, such as all bounded domains, and domains exhibiting polynomial (resp. integrable) decay in horizontal (resp. vertical) directions. In particular, this class includes cases with periodic and image boundary conditions, both of which are of considerable importance in the study of free-boundary problems.

% In class \ref{def:case Omega=S}, the strip can be regarded as a direct extension of the half-plane model in \cite{abe2022stability}, or as a generalization of the class \ref{def:case Omega=D} to $p=\infty$. However, due to the different widths of the strips, i.e. the sizes of $L$ vary, directly applying the method in half-plane model will encounter some difficulties comes form geometric properties. Meanwhile, due to the disappearance of the decay of the domain itself, the methodology in previous class is also difficult to hold.

More precisely, we introduce the set of admissible vorticities for parameters $0\leq \mu,\nu,\lambda < \infty$:
\begin{align*}
    K_{\mu,\nu}=\left\{\omega\in L^2(\Omega) \big| \omega\geq0, \int_{\Omega}x_2\omega(x)dx=\mu, \int_{\Omega}\omega(x)dx\leq\nu\right\}.
\end{align*}

The kinetic energy and the penalized energy are given by
\begin{align*}
    E[\omega]=\frac{1}{2}\int_{\Omega}\int_{\Omega}G(x,y)\omega(x)\omega(y)dxdy,\quad E_{2,\lambda}[\omega]=E[\omega]-\frac{1}{2\lambda}\|\omega\|_{2}^2,
\end{align*}
where $G$ is the Green’s function for the domain $\Omega$. Define $S_{\mu,\nu,\lambda}\subseteq K_{\mu,\nu}$ by the set of minimizers for \eqref{def:minimizing problem} in $K_{\mu,\nu}$. By the measure scaling, $\hat{\omega}(x)=\omega(x)/\nu,$ the problem reduces to the case $\nu=1$. We abbreviate the notation as 
\begin{align*}
    K_\mu=K_{\mu,1}, ~ I_{\mu,\lambda}=I_{\mu,1,\lambda}, \tt{ and } S_{\mu,\lambda}=S_{\mu,1,\lambda}.
\end{align*}
Recall the the stream function associated with vorticity $\omega$ in \eqref{def:stream function},
% \begin{align}
%     \psi(x)=\int_{\Omega} G(x,y) \omega(y)dy, \label{def:stream function}
% \end{align}
$E[\omega]=\frac{1}{2}\int \psi\omega dx$. Adapting the proof of \cite{abe2022stability} to our setting encounters three main difficulties that demand new analytical tools.

\subsection{Loss of spatial scaling.}  \label{sec: difficulty_loss_scaling}
  % In~\cite{abe2022stability} the half-plane is invariant under the scaling $\hat\omega(x)=(\lambda\nu)^{-1}\omega(\lambda^{-1/2}x)$, which reduces the problem to the case $\lambda=1$.  For a general domain from Class~1 or Class~2, the scaled function no longer lives in the same domain, and the kinetic energy does not transform simply.  Hence we cannot unify $\lambda$.  We overcome this by using the maximum principle to obtain the pointwise estimate $0\le G(x,y)\le G_H(x,y)$ (where $G_H$ is the half-plane Green's function), which allows us to mimic the half-plane estimates and to prove negativity of $I_{\mu,\lambda}$ by constructing explicit trial functions (Proposition~\ref{prop:negativity}).
    Due to the loss of spatial scaling invariance, the Green's function on $\Omega$ does not satisfy the property that $G(cx,cy)=G(x,y)$. Therefore, for any $c>0$ and scaling $\widehat\omega (x)=\omega(cx)$, there is no possible connection between $E[\widehat\omega|_{\Omega}]$ and $E[\omega]$. Meanwhile, if we consider the synchronous changes of the domain $\Omega$ during scaling and regard $\widehat\omega$ as a function defined on $c^{-1}\Omega$, the penalized energy of $\widehat\omega$ should be:
    \begin{align*}
        E^c_{2,\lambda}[\widehat\omega] = & \int_{c^{-1}\Omega}\int_{c^{-1}\Omega}G^{c}(x,y)\widehat\omega(x)\widehat\omega(y)dxdy- \frac{1}{2\lambda}\int_{c^{-1}\Omega}\widehat\omega^2dx\\
        = & c^{-4}\int_{\Omega}\int_{\Omega}G(x,y)\omega(x)\omega(y)dxdy-\frac{1}{2\lambda}\int_{c^{-1}\Omega}\widehat\omega^2dx\\
        = & c^{-4}E[\omega] - c^{-2}\frac{1}{2\lambda}\int_{\Omega}\omega^2dx,
        % \\
        % E^{c}[\hat\omega]=&\int_{c^{-1}\Omega}\int_{c^{-1}\Omega}G^{c}(x,y)\hat\omega(x)\hat\omega(y)dxdy\\
        % =&c^{-4}\int_{\Omega}\int_{\Omega}G(x,y)\omega(x)\omega(y)dxdy=c^{-4}E[\omega_{a,b}],\\
        % \|\hat\omega\|_2^2=&\int_{c^{-1}\Omega}\hat\omega^2dx=c^{-2}\int_{\Omega}\omega^2dx,
    \end{align*}
    where we temporarily use $E^{c}$ and $G^c$ to represent the kinetic energy and the Green's function for $c^{-1}\Omega=\left\{c^{-1}x:x\in \Omega \right\}$. To compare $E_{2,\lambda}^{c}[\widehat\omega]$ with $E_{2,\lambda}[\omega]$, it is necessary to let $c^{-4}=c^{-2}$ with an unique solution $c=1$, which is meaningless. 

    This issue directly led to the failure of the methods used in \cite{abe2022stability} to study kinetic energy and penalized energy. Fortunately, based on the inclusion relationship $\Omega\subseteq \R^2_+$, we have
    \begin{align}
        0\leq G(x,y)\leq G_H(x,y) \tt{ for any } x,y\in \overline{\Omega},\label{ineq:estimate of green's function}
    \end{align}
    where $G_H$ is the Green's function for the upper half plane, and this comparison allows us to derive estimates similar to those in \cite{abe2022stability}. By choosing appropriate parameters $\mu,\lambda$ and constructing explicit trial functions in $K_{\mu,\lambda}$, we prove negativity of \(I_{\mu,\nu,\lambda}\). Other impacts of the loss of scaling will be mentioned and solved below.

\subsection{Loss of horizontal translation invariance (Class~\ref{def:case Omega=D}).}  
    % The Green's function of a general decay domain is not a function of $|x_1-y_1|$, so the kinetic energy is not translation invariant.  Rearrangement and concentration arguments based on translation, therefore, fail.  We exploit the decay of $\mathcal{D}$ itself to show that the weighted mass of any minimising sequence concentrates on a bounded region $D_0$.  On $D_0\times D_0$ the Green's function is square-integrable, which allows us to pass to the limit in the kinetic energy using weak $L^2$ convergence and to obtain the existence of a minimiser.
    The Green's function of domains in this class
    % meet the weak finite volume condition 
    cannot be written as $G(x,y)=g(|x_1-y_1|,x_2,y_2)$ for some function $g$, so that the kinetic energy is not translation invariant. Meanwhile, such difficulty also undermines all techniques that rely on rearrangement methods. 
  
    To overcome this, we introduce the weak finite volume condition of $\D$ itself to show that the weighted mass of any minimizing sequence concentrates on a bounded region. In particular, for any minimizing sequence $\{\omega_n\}$ and $\epsilon>0$, there exists a bounded domain $D_0\subseteq \D$ such that the kinetic energy, as well as crossing terms, involving $\omega_n|_{\D\setminus D_0}$ is in $\mathcal{O}(\eps)$. Under such a split, the Green's function is square-integrable on $D_0\times D_0$, so that we can pass the limit in kinetic energy using weak $L^2$ convergence and obtain the existence of a minimizer and a general convergence theorem.

\subsection{Difficulty in proving strict monotonicity of $I_\mu$ (Class~\ref{def:case Omega=S}).}  
    % In~\cite{abe2022stability} the scaling invariance immediately yields the strict decrease of $I_{\mu,\lambda}$ with respect to $\mu$, a crucial ingredient for the concentration-compactness argument.  Because we lack scaling invariance, we cannot follow the same route.  Instead, we reverse the logic: we first prove the existence of a minimiser (by a symmetrisation method and careful energy estimates, see \S\ref{sec:exists of minimizer in strip}), and then use the properties of this minimiser (uniform $L^1$ bound, compact support) to establish the strict monotonicity of $I_{\mu,\lambda}$ (Proposition~\ref{prop:strict_dec}).  This strict monotonicity is finally employed in a concentration-compactness framework to obtain a general convergence theorem for minimising sequences (Theorem~\ref{thm:gen_conv_S}).

    The method in \cite{abe2022stability}, as well as \cite{alma991005089079703407} and \cite{FriedmanAvner1981Vrea}, applies to prove compactness of a minimizing sequence satisfying
    \begin{align}
        \begin{aligned}
            & \omega(x_1,x_2)=\omega(-x_1,x_2),\\
            & \omega(x_1,x_2) \tt{ is non-increasing for } x_1 > 0.
        \end{aligned}\label{def:symmetry property}
    \end{align}
    Under such an assumption, these authors discuss the weighted mass of the minimizing sequence of $I_{\mu,\nu,\lambda}$ is concentrated under a proper $x_1$-translation, and the key to this lies in the rearrangement and the strictly decreasing nature of $I_{\mu,\nu,\lambda}$. The former can be resolved by two points: i) the structure of the Green's function for $\S$ meet the criteria of standard rearrangement inequality (see Proposition \ref{prop:steiner symmetrization}), ii) the comparison \eqref{ineq:estimate of green's function} leads to a comparison between stream functions by viewing $L^i(\S)=\{\omega\in L^i(\R_+^2) ~ | ~ \spt\omega\subseteq \overline{\S} \}$, then for any proper $i>1$ and $\omega\in L^i(\S)$ can generate a stream function $\psi^H$ on $\R^2_+$ in the sense that
    \begin{align*}
        \psi^H(x)=\int_{\R^2_+}G_H(x,y)\omega(y)dy=\int_{\S}G_H(x,y)\omega(y)dy,
    \end{align*}
    with 
    \begin{align}
        \psi(x)\leq \psi^H(x) \tt{ for any } x\in \S. \label{ineq:comparsion between stream function}
    \end{align}
    The latter is more complicated due to the lack of scaling invariance. We cannot follow the same route, instead, it is necessary to proof it alone a new path: prove the existence of a minimizer (see Sec \ref{sec:exists of minimizer in strip}), and then demonstrate more properties of minimizers as $\mu$ is sufficiently small, such as an uniform $L^1$ bound, to establish the strict monotonicity of $I_{\mu,\nu,\lambda}$ (Proposition~\ref{prop:decreasing of I_mu}).  This strict monotonicity is finally employed in a concentration-compactness framework to obtain a general convergence theorem for minimizing sequences (Theorem~\ref{thm:general convergence theorem for strip}).

\subsection{Organization of the paper.}
The paper is organized as follows. Section \ref{sec:minimizing problem} contains preliminary estimates, Proposition \ref{prop:estimate for kinetic energy} and \ref{prop:propoties of I_mu}, establishing the negativity of $I_{\mu,\lambda}$ for suitable parameters and the uniform $L^2$ boundedness of minimizing sequences. Consequently, we investigate the general structure of the minimizer in Proposition \ref{prop:general structure}.

Section \ref{sec:general convergence for class 1} is devoted to a general convergence theorem to the variational problem \eqref{def:minimizing problem} for the Class \ref{def:case Omega=D}, including the discussion about the kinetic energy (Proposition \ref{prop:weak_finite_volume_energy_convergence}), and the existence of minimizer from minimizing sequence (Theorem \ref{thm:exists of minimizer}).
%, which are all ingredients for the general convergence theorem. 
Section \ref{sec:general convergence for class 2} is devoted to a general convergence theorem for the Class \ref{def:case Omega=S}, including the discussion about the existence of minimizer (Sec \ref{sec:exists of minimizer in strip}), the uniformly $L^1$-norm decay to minimizers (Proposition \ref{prop:minimizer L1,2 estimate}), and the general convergence theorem (Theorem \ref{thm:general convergence theorem for strip}).

In Section \ref{sec:stability}, we prove the orbital stability of the minimizer over the lifespan (Theorem \ref{thm:stable wrt minimizer}); finally, in Section \ref{sec:discussion}, we present a discussion of future work.

\subsection*{Acknowledgement}
Both authors acknowledge support from Hong Kong RGC Grants CUHK--14304424 and CUHK--14301225. 

\section{A Minimization Problem}\label{sec:minimizing problem}

We establish the basic properties of the minimization problem. Thanks to the observation \eqref{ineq:estimate of green's function}, we derive several useful bounds for the kinetic energy and crossing terms. These are collected in Proposition \ref{prop:estimate for kinetic energy} and are analogous to those obtained in \cite{abe2022stability}. Next, we study the sign of $I_{\mu,\lambda}$. The proof of negativity differs from that in \cite{abe2022stability} because we lack scaling invariance; instead, we construct an explicit trial function supported on a finite measure subset of $\Omega$. See Proposition \ref{prop:propoties of I_mu} below.

\begin{prop}\label{prop:estimate for kinetic energy}
    The following estimates hold for $\omega,\omega_i\in L^1\cap L^2(\Omega)$ satisfying $x_2\omega, x_2\omega_i\in L^1(\Omega)$ with some constant $C$ independent from $\omega,\omega_i, i=1,2$:
    \begin{align}
        &\Bigg|\int_{\Omega}G(x,y)\omega(y)dy\Bigg|\leq C x_2^{1/2}\|\omega\|_{1}^{1/2}\|\omega\|_{2}^{1/2},\label{ineq:estimate of stream fct}\\
        &E[\omega]\leq C \|x_2\omega\|_{1}^{1/2}\|\omega\|_{1}\|\omega\|_{2}^{1/2},\label{ineq:estimate of kinetic energy}\\
        &\Bigg|\int_{\Omega}\int_{\Omega}G(x,y)\omega_1(y)\omega_2(x)dxdy\Bigg|\leq C \|\omega_1\|_{1}^{1/2}\|\omega_1\|_{2}^{1/2}\|x_2\omega_2\|_{1}^{1/2}\|\omega_2\|_{1}^{1/2},\label{ineq:estimate of energy crossing term}\\
        &|E[\omega_1]-E[\omega_2]|\leq C \|\omega_1-\omega_2\|_{1}^{1/2}\|\omega_1-\omega_2\|_{2}^{1/2}\|x_2(\omega_1+\omega_2)\|_{1}^{1/2}\|\omega_1+\omega_2\|_{1}^{1/2}.\label{ineq:estimate of energy difference}
    \end{align}
\end{prop}
\noindent\textit{Proof: } Use the inequality $ 0\leq G(x,y)\leq G_H(x,y)$, we derive the bound for any $q\in(1,2)$ that
\begin{align*}
    \Bigg(\int_{\Omega}G(x,y)^qdy\bigg)^{1/q}\leq \Bigg(\int_{\R^2_+}G_H(x,y)^qdy\Bigg)^{1/q}\leq Cx_2^{2/q}.
\end{align*}
Then all estimates \eqref{ineq:estimate of stream fct}-\eqref{ineq:estimate of energy difference} can be obtained through the same arguments in Proposition 2.1, \cite{abe2022stability}.

\rightline{$\Box$}

These estimates, while elementary, are fundamental for the rest of the paper. They not only provide uniform bounds for general minimizing sequences (as will be seen in Remark \ref{remark:uniformly bounded}) but also serve as the workhorse for controlling error terms in convergence arguments, such as the proof of Proposition \ref{prop:weak_finite_volume_energy_convergence} and \ref{prop:strip energy_convergence}. 

Next, we shall prove that $I_{\mu,\lambda}$ is negative for a suitable class of parameters $(\mu,\lambda)$. 
Different from Lemma 2.3 in \cite{abe2022stability}, where $I_{\mu,\lambda}<0$ for all positive $\mu$ and $\lambda$, we can only prove the negativity for suitable parameters.
% This is in contrast to Lemma 2.3 in \cite{abe2022stability}, where the negativity of $I_{\mu,\lambda}$ for all $\mu,\lambda>0$ follows from the spatial scaling invariance of $\R_+^2$. Since our domain lacks this property, a different argument is required.

\begin{prop}\label{prop:propoties of I_mu} Recall the definition of $I_{\mu,\lambda}=I_{\mu,1,\lambda}$ in \eqref{def:minimizing problem}
    \begin{align}
        & I_{0,\lambda}=0, && \text{ for any $\lambda\in \R$}, \label{eqn:I_0}\\
        & I_{\mu,\lambda}>-\infty, && \text{ for any } \mu\geq 0 \text{ and } \lambda>0 ,\label{eqn:I_mu is finite}\\
        & I_{\mu,\lambda}<0, && \text{ for any } (\mu,\lambda)\in \bigcup_{K\subseteq \Omega \text{ and } |K|<\infty}
         (0, \mu_K ] \times (\lambda_K,\infty ). \label{ineq:negativity of I_mu}
    \end{align}
    where $\mu_K=\frac{1}{\tt{Vol(K)}} \int_K x_2 dx$ and $\lambda_K=\frac{\tt{Vol(K)}}{\int_K\int_K G(x,y) dydx}$.
\end{prop}
\noindent\textit{Proof: } The property \eqref{eqn:I_0} is trivial since $K_0=\{0\}$. By \eqref{ineq:estimate of kinetic energy} and Young's inequality, for any $\omega\in K_\mu$,
\begin{align*}
    E_{2,\lambda}[\omega]&\leq C \|x_2\omega\|_{1}^{1/2}\|\omega\|_{1}\|\omega\|_{2}^{1/2}-\frac{1}{2\lambda}\|\omega\|_2^2\\
    &\leq \frac{3}{4}\Bigg(2^{-1/4}C\lambda^{1/4}\|x_2\omega\|_{1}^{1/2}\|\omega\|_{1}\Bigg)^{4/3}+\Bigg(\frac{2}{4\lambda}-\frac{1}{2\lambda}\Bigg)\|\omega\|_2^2\\
    &\leq C\lambda^{1/3}\mu^{2/3},
\end{align*}
then $$I_{\mu,\lambda}=\inf_{\omega\in K_{\mu,\lambda}}\big\{-E_2[\omega ]\big\}\geq -C\lambda^{1/3}\mu^{2/3}>-\infty.$$

For any $(\mu,\lambda)\in
        (0,\mu_K ] \times (\lambda_K,\infty)$ where $K\subseteq \Omega \text{ and } |K|<\infty$, define
\begin{align*}
    \omega_0=c_0 1_{K}, \text{ where } c_0=\mu\Big(\int_{K}x_2dx\Big)^{-1}>0.
\end{align*}
It is easy to see $\omega_0\in K_\mu $ by the choice of $\mu$, and 
\begin{align*}
    E_{2,\lambda}[\omega_0]&= \frac{1}{2}c_0^2\int_{K}\int_{K}G(x,y)dxdy-\frac{1}{2\lambda}c_0^2 \tt{Vol}(K)\\
    &=\frac{1}{2}c_0^2\Big(\int_{K}\int_{K}G(x,y)dxdy-\frac{1}{\lambda}\tt{Vol}(K)\Big)\\
    &>0,
\end{align*}
where the last inequality is positive by the choice of $\lambda$.

\rightline{$\Box$}

\begin{remark}[For general $\nu>0$]\label{remark:general nu}
\end{remark}
In this case, according to the scaling $\hat\omega=\omega/\nu$, the ranges of $\mu$ and $\lambda$ in \eqref{ineq:negativity of I_mu} should be 
\begin{align}
    (\mu,\lambda)\in \bigcup_{K\subseteq \Omega \text{ and } |K|<\infty}
        (0, \mu_K \nu] \times (\lambda_K, \infty).\label{ineq:negativity of I_mu, general nu}
\end{align}

\begin{remark}[Minimizing Sequence is $L^2$ Bounded]\label{remark:uniformly bounded}
\end{remark}
Any minimizing sequence $\{\omega_n\}$ satisfying $\omega_n\in K_{\mu_n}, \mu_n\rightarrow \mu$, and $-E_{2,\lambda}[\omega_n]\rightarrow I_\mu$ with $\mu, \lambda$ follows the assumption in \eqref{ineq:negativity of I_mu} is uniformly bounded in $L^2$. Indeed, by \eqref{ineq:estimate of stream fct} and Young’s inequality, for arbitrary $\epsilon >0 $ and $\omega\in K_\mu$,
    \begin{align*}
        \frac{1}{2\lambda}\|\omega\|_{2}^2+E_{2,\lambda}[\omega]=E[\omega]\leq \frac{3}{4}\Big(\frac{C}{\epsilon^{1/2}}\|x_2\omega\|_1^{1/2}\|\omega\|_1\Big)^{4/3}+\frac{\epsilon^2}{4}\|\omega\|_2^2.
    \end{align*}
    By taking $\epsilon=\lambda^{-1/2}$, 
    \begin{align*}
        \|\omega\|_{2}^2\leq C\mu^{2/3}\lambda^{4/3}\|\omega\|_1^{4/3}-4\lambda E_{2,\lambda}[\omega].
    \end{align*}
    Thus by $I_{\mu,\lambda}<0$, the minimizing sequence satisfies 
    \begin{align*}
        \limsup_{n\rightarrow +\infty}\|\omega_n\|_{2}\leq C\mu^{1/3}\lambda^{2/3}\limsup_{n\rightarrow +\infty}\|\omega_n\|_1^{2/3}. 
        % \label{ineq:uniformly bounded}
    \end{align*}
    In particular, if $\omega$ is a minimizer, then 
    \begin{align}
        \|\omega\|_{2}\leq C\mu^{1/3}\lambda^{2/3}\|\omega\|_1^{2/3}. \label{ineq:uniformly bounded}
    \end{align}

\begin{remark}[Behavior of $\mu,\lambda$]\label{remark:I decreasing in mu}
\end{remark}
In \eqref{ineq:negativity of I_mu}, the range of $\mu, \lambda$ is highly related to the geometric structure of $\Omega$. For those domains with finite areas, one can directly take $K=\Omega$ and simplify accordingly. In general, a suggestion is to let $K\subseteq\Omega$ be a rectangle centered at $x_0$ with proper lengths and widths of $a$ and $b$, respectively. Doing so, the range of $\mu$ determined by this $K$ can be directly calculated, and a subrange of $\lambda$ can be obtained by using the Green's function $G_{a,b}$ for the rectangle and the observation $G_{a,b}\leq G$, then a lower bound to $\int_K\int_K G(x,y)dxdy$ can be dervied.
% $G_{a,b}\leq G$, where $G_{a,b}$ is the green's function on the rectangle. Finally, the union can be calculated.

% In the Lemma 2.3 of \cite{abe2022stability}, the authors used scaling invariance to prove that for fixed $\lambda$, the quantity $-I_{\mu,\lambda}$ is strictly increasing in $\mu$. This reflects the fact that in $\R_+^2$, horizontal slices at each height are identical, so a higher vertical centroid corresponds to larger energy. However, in domains with an upper bound, such as a horizontal strip or the half-strip considered in \cite{choi2022stability}, the presence of 
% an upper boundary prevents the energy from necessarily increasing with $\mu$, even though horizontal slices remain identical. In our setting, the domain $\mathcal{D}$ is more complicated: although it is not necessarily bounded above, the horizontal slices at different heights are not identical, so the energy depends on the vertical centroid in a non‑monotone way. The loss of this monotonicity poses an additional difficulty beyond the failure of rearrangement techniques. Fortunately, as mentioned in the introduction, this obstacle can be overcome by employing operator methods.

\begin{prop}[General structure of minimizers]
\label{prop:general structure}
Let $\mu,\lambda$ be given by the assumption in \eqref{ineq:negativity of I_mu}. Each minimizer $\omega \in S_{\mu,\lambda}$ satisfies
\begin{equation}
\label{eqn:general structure of minimizer}
\omega = \lambda(\psi -W x_{2} - \gamma)_+, \qquad 
\psi (x) = \int_{\Omega}G(x,y)\omega (y)\mathrm{d}y,
\end{equation}
for some constants $W, \gamma \in \R$, uniquely determined by $\omega$.
\end{prop}
\noindent\textit{Proof: }
The proof follows from a standard argument, e.g., \cite{friedman1982variational, friedman1981vortex} for vortex rings, and is a direct consequence of Proposition 2.5 in \cite{abe2022stability}, we omit the proof. 

\rightline{$\Box$}

\begin{remark}[Measure and $\gamma$]\label{remark:measure and gamma for minimizer}
Every minimizer in $\omega\in S_{\mu,\lambda}$ s.t. $\gamma \neq 0 $, we have $$\int \omega dx=1.$$
\end{remark}
\begin{proof}
    Exactly the same as the Remark 2.6(ii) in \cite{abe2022stability}

\end{proof} 

\section{General Convergence Theorem for the Class \ref{def:case Omega=D}}\label{sec:general convergence for class 1}

With the preliminary estimates and the negativity of $I_{\mu,\lambda}$ established, the next step is to prove that a general minimizing sequence actually converges to a minimizer (Theorem \ref{thm:exists of minimizer}). To achieve this goal, it is necessary to study the convergence of the kinetic energy $E[\omega]$ under weak $L^2$ convergence.

\begin{prop}[Convergence of the kinetic energy]
\label{prop:weak_finite_volume_energy_convergence}
Let \(\mathcal{D}\) be a domain in \(\R^2_+\) satisfy the weak finite volume condition. Assume \(\{\omega_n\} \subset L^1  \cap L^2(\mathcal{D})\) and \(\omega \in L^2(\mathcal{D})\) satisfy:
\begin{itemize}
    \item \(\sup_n \|\omega_n\|_1 + \sup_n \|\omega_n\|_2 + \|\omega\|_1+ \|\omega\|_2\le M\) for some constants \(M\), and
    \item \(\omega_n \rightharpoonup \omega\) weakly in \(L^2(\mathcal{D})\).
\end{itemize}
Then for any $\epsilon>0$, there exists $N=N(\epsilon)>0$, such that 
\begin{align}
    \sup_{n}\int_{D_l \cap \{|x_1|>N\}} x_2 \omega_n dx\leq \epsilon, \label{ineq:concentration of bdd seq}
\end{align}
hence 
\(E[\omega_n] \to E[\omega]\).
\end{prop}
\noindent\textit{Proof:}
Suppose that \eqref{ineq:concentration of bdd seq} is false. Then there exists $\epsilon_0>0$ and a subsequence $\{\omega_{n_k}\}$ such that for any $k \in \mathbb{Z}_{\geq 1}$,
\begin{align*}
    \int_{D_l \cap \{|x_1|>k\}} x_2 \omega_{n_k} dx > \epsilon_0.
\end{align*}

Thus 
\begin{align*}
    \epsilon_0 < \int_{D_l \cap \{|x_1|>k\}} x_2 \omega_{n_k} dx \leq  \Bigg(\int_{D_l \cap \{|x_1|>k\}} x_2^p dx\Bigg)^{1/p}\Bigg(\int_{D_l \cap \{|x_1|>k\}} \omega_{n_k}^q dx\Bigg)^{1/q}.
\end{align*}
where $\frac{1}{p}+\frac{1}{q}=1$ with $p\geq 2$, we see $q\in (1,2]$. By the weak finite volume condition \eqref{def:weak_finite_volume_domain_assumption} for $D_l$,
\begin{align*}
    \int_{D_l \cap \{|x_1|>k\}} x_2^p dx \rightarrow 0, \tt{ as } k \rightarrow \infty.
\end{align*}
Then 
\begin{align*}
    \|\omega_{n_k}\|_{q} \geq (\int_{D_l \cap \{|x_1|>k\}} \omega_{n_k}^q dx)^{1/q} \geq \epsilon_0 \Big(\int_{D_l \cap \{|x_1|>k\}} x_2^p dx\Big)^{-1/p}\rightarrow +\infty.
\end{align*}
However, according to our assumption: $\|\omega_{n_k}\|_{1}+ \|\omega_{n_k}\|_{2}\leq M$, then $\|\omega_{n_k}\|_{q}\leq 2M$ for any $q\in (1,2]$, a contradiction arises.

Based on \eqref{ineq:concentration of bdd seq}, for any $\epsilon>0$, and define $$
D_0=D_u\cup(D_l \cap \{|x_1|\leq N\}),
$$
where $N=N(\epsilon^2)$ is the constant given in the previous step. By properly increasing $N$, assume
\begin{align*}
    \int_{\D\backslash D_0} x_2\omega dx\leq \epsilon^2.
\end{align*}
Then 
\begin{align*}
    2|E[\omega_n] & - E[\omega]| \leq 2|E[\omega_n|_{D_0}]-E[\omega|_{D_0}]| \\
     & + 2\int_{\mathcal{D}}\int_{\mathcal{D}} G(x,y)\omega_n|_{D_0}(x)\omega_n|_{\mathcal{D}\backslash D_0}(y)dydx + E[\omega_n|_{\mathcal{D}\backslash D_0}]\\
     & + 2\int_{\mathcal{D}}\int_{\mathcal{D}} G(x,y)\omega|_{D_0}(x)\omega|_{\mathcal{D}\backslash D_0}(y)dydx + E[\omega|_{\mathcal{D}\backslash D_0}].
\end{align*}

For last four terms, by \eqref{ineq:estimate of kinetic energy}, \eqref{ineq:estimate of energy crossing term} and the fact that $\sup_n\|x_2\omega_n|_{\mathcal{D}\backslash D_0}\|_1\leq \epsilon^2$, $\|x_2 \omega|_{\mathcal{D}\backslash D_0}\|_1\leq \epsilon^2$, then 
\begin{align*}
    |E[\omega_n] & - E[\omega]| \leq |E[\omega_n|_{D_0}]-E[\omega|_{D_0}]| + CM^2\epsilon.
\end{align*}

For the first term, since $G(x,y) \in L^2(D_0 \times D_0)$ and $\omega_n(x)\omega_n(y) \rightharpoonup \omega(x)\omega(y)$ in $L^2(D_0 \times D_0)$, sending $n\to\infty$
\begin{align*}
    E[\omega_n|_{D_0}] = & \int_{D_0}\int_{D_0} {G}(x,y)\omega_n(x)\omega_n(y)dydx \\
    & \longrightarrow \int_{D_0}\int_{D_0} {G}(x,y)\omega(x)\omega(y)dydx = E[\omega|_{D_0}].
\end{align*}

In conclusion, as $n$ is sufficiently large, 
\begin{align*}
    |E[\omega_n] & - E[\omega]| \leq 2CM^2\epsilon\rightarrow 0.
\end{align*}

\rightline{$\Box$}

Proposition \ref{prop:weak_finite_volume_energy_convergence} provides a crucial ingredient: the kinetic energy is continuous with respect to weak convergence in \(L^2\) provided the sequence satisfies a uniform bound. This allows us to pass to the limit in the minimization problem, hence proving the existence of a minimizer and a general convergence theorem.

\begin{thm}\label{thm:exists of minimizer}
    Let $\nu>0$, and $\mu,\lambda$ satisfy \eqref{ineq:negativity of I_mu, general nu}. For any minimizing sequence $\{\omega_n\}$ satisfying $\omega_n\in K_{\mu_n,\nu}$, $\mu_n \rightarrow \mu $ and $-E_{2,\lambda}[\omega_n] \rightarrow I_{\mu,\nu,\lambda}$, then there exists a subsequence, still denoted by $\{\omega_n\}$, such that there exists $\omega\in K_{\mu,\nu}$, 
    $\omega_n\rightarrow\omega$ and $x_2\omega_n\rightarrow x_2\omega$ strongly in $L^2(\mathcal{D})$ and $L^1(\mathcal{D})$ respectively. In particular $\omega\in S_{\mu,\nu,\lambda}$ and hence $S_{\mu,\nu,\lambda}\neq \emptyset$.
\end{thm}
\noindent\textit{Proof:} It is sufficient to prove the case $\nu=1$. Let $\{\omega_n\}$ be a minimizing sequence such that $\omega_n\in K_{\mu_n}$, $\mu_n \rightarrow \mu $ and $-E_{2,\lambda}[\omega_n] \rightarrow I_{\mu,\lambda}$ as $n\rightarrow+\infty$. Since $\{\omega_n\}$, $\{x_2\omega_n\}$ are uniformly bounded in $L^2$, $L^1$ respectively, then there exists a subsequence still denoted by $\{\omega_n\}$, $\omega_n\rightharpoonup\omega, x_2\omega_n\rightharpoonup x_2 \omega$ in $L^2$, $L^1$ respectively for some $\omega$. We shall proof that $\omega$ is a minimizer of $I_{\mu,\lambda}=\inf_{\omega\in K_{\mu}}\{-E_{2,\lambda}[\omega]\}$ in two steps. 

\underline{\textit{Step 1:}} $\omega\in K_{\mu}$.

For any $R>0, $ since $ B(0,R)\cap \mathcal{D}$  has finite measure, then by weak convergence, 
$$\int_{B(0,R)\cap \mathcal{D}}\omega dx=\lim_{n\rightarrow\infty}\int_{B(0,R)\cap \mathcal{D}}\omega_n dx\leq 1,$$
by the Monotone Convergence Theorem, 
$$\int_{\mathcal{D}}\omega dx=\lim_{R\rightarrow\infty}\int_{B(0,R)\cap \mathcal{D}}\omega dx\leq 1.$$

In addition to this, by $x_2\omega_n$ converge weakly to $x_2\omega$ in $L^1$ and $1_{\mathcal{D}}\in  L^{\infty}(\mathcal{D}),$
\begin{align*}
    \int x_2 \omega_n \cdot 1_{\mathcal{D}} dx=\mu_n\rightarrow \mu=\int x_2 \omega_n \cdot 1_{\mathcal{D}} dx,
\end{align*}
thus $\omega \in K_{\mu}$.

\underline{\textit{Step 2:}} Strong convergence and conclude the proof.

Recalling that $\{\omega_n\}$ is uniformly bounded in $L^2$ and $L^1$, $\omega_n$ converge weakly to $\omega$ in $L^2$, then this sequence meet the assumption in the Proposition \ref{prop:weak_finite_volume_energy_convergence}, see $E[\omega_n] \rightarrow E[\omega]$. Combining with Step 1, 
\begin{align}
    -I_{\mu,\lambda}=\lim_{n\rightarrow\infty}E_{2,\lambda}[\omega_n]\leq \lim_{n\rightarrow\infty}E[\omega_n]-\frac{1}{2\lambda}\liminf_{n\rightarrow\infty}\|\omega_n\|_{2}^2 \leq E_{2,\lambda}[\omega]\leq -I_{\mu,\lambda} \label{eqn:conclusion of claim 1}
\end{align}

Thus 
\begin{align*}
    -E_{2,\lambda}[\omega] = I_{\mu,\lambda}, && \|\omega\|_{2}=\lim_{n\rightarrow\infty}\|\omega_n\|_{2},
\end{align*}
the former shows $\omega$ is a minimizer of $I_{\mu,\lambda}=\inf_{\omega\in K_{\mu}}\{-E_{2,\lambda}[\omega]\}$ in $K_{\mu}$, the later show $\omega_n$ converge to $\omega$ strongly in $L^2$. 

In addition to that, to prove the strong $L^1$ convergence of $x_2\omega_n$, we separate $\mathcal{D}$ into two subdomains: $D_0, $ and $ \mathcal{D}\backslash D_0$, where $D_0$ is the set used in the Proposition \ref{prop:weak_finite_volume_energy_convergence}.

For the second part, by the choice of $N$ and triangle inequality,
\begin{align*}
    \sup_n\int_{\D\backslash D_0}x_2|\omega_n-\omega|dx \leq 2\epsilon. 
\end{align*}

For the first part, by $|D_0|<\infty$ and the strong $L^2$ convergence of $\omega_n$, $\omega_n|_{D_0}$ converge to $\omega|_{D_0}$ strongly in $L^1$. There exists a subsequence, still denoted by $\omega_n$ s.t. $\omega_n|_{D_0}$ converge to $\omega|_{D_0}$ pointwisely. By Fatou's lemma:
\begin{align*}
    2\int_{D_0} x_2\omega ~ dx \leq & \liminf_n \int_{D_0} (x_2\omega_n+x_2\omega-x_2|\omega_n-\omega|)dx,\\
    \leq & \mu+ \int_{D_0} x_2\omega ~ dx-\limsup_n \int_{D_0} x_2|\omega_n-\omega|dx,
\end{align*}
then by the choice of $D_0$
\begin{align*}
    \mu =& \int_{D_0} x_2\omega ~ dx, +\int_{\D\backslash D_0} x_2\omega ~ dx\\
    \leq & \mu + \eps -\limsup_n \int_{D_0} x_2|\omega_n-\omega|dx,
\end{align*}
so that,
\begin{align*}
    \limsup_n \int_{D_0} x_2|\omega_n-\omega|dx\leq \eps.
\end{align*}

Hence, combine these two parts,
\begin{align*}
    \limsup_n \int_{\D} x_2|\omega_n-\omega|dx\leq 3\eps,
\end{align*}
while the left hand side is independent from $\eps$. Passing $\eps\to 0$, we conclude $x_2\omega_n\rightarrow x_2 \omega$ in $L^1(\mathcal{D})$.

\rightline{$\Box$}

\section{General Convergence Theorem for the Class \ref{def:case Omega=S}}\label{sec:general convergence for class 2}

\subsection{Existence of Minimizer}\label{sec:exists of minimizer in strip}

To obtain the same convergence of kinetic energy as \cite{abe2022stability}, it is necessary to provide a strictly negative upper bound for $I_{\mu,\lambda}$ and demonstrate that the variational problem \eqref{def:minimizing problem} admits a minimizer.
% To apply the method in \cite{abe2022stability}, which based on the strictly decreasing of $I_{\mu,\lambda}$ (Proposition \ref{prop:decreasing of I_mu}) and the concentration-compactness principle later, it is necessary to ensure that $I_{\mu,\lambda}$ is negative and the variational problem \eqref{def:minimizing problem} admits a minimizer. 
As noted in Proposition \ref{prop:propoties of I_mu}, the negativity of $I_{\mu,\lambda}$ depends on finding a suitable trial function supported on a set $K \subset \mathcal{S}$. 

For our specific domain $\S$, it is sufficient to consider $K=(0, 10L) \times (0, L)$, which comfortably fits within the infinite strip $\S=\mathbb{R}\times(0,L)$. The primary motivation for this choice is to obtain a concrete and strictly positive lower bound for $-I_{\mu, \lambda}$ using the explicit Green's function \eqref{def:green's function for strip} and the choice of $(\mu,\lambda)$. 
% We now estimate the penalized energy of the trial function $\omega_1 = c_1 \mathbf{1}_{K}$ under a proper range of parameters $(\mu, \lambda)$ for which $I_{\mu, \lambda} < 0$. 
Indeed, we take
\begin{align}
    (\mu,\lambda)\in (0,L/2) \times [\frac{\pi^5}{2L^2},\infty). \label{ineq:negativity of I_mu, strip}
\end{align}
with a trial function
\begin{align*}
    \omega_0=\frac{\mu}{5L^3} 1_{(0,10L)\times (0,L)}=c_0 1_{(0,10L)\times (0,L)},
\end{align*}
so that, 
\begin{align*}
    &-I_{\mu,\lambda} \geq  E_{2,\lambda }[\omega_0]  \\
    = & \frac{c_0^2}{2}\Bigg( \int_{(0,10L)\times (0,L)}\int_{(0,10L)\times (0,L)}G(x,y)dxdy - \frac{10L^2}{\lambda}\Bigg)\\
    = & \frac{c_0^2}{2}\Bigg( \frac{1}{4\pi }\int_{K}\int_{K}\ln\Big(1+\frac{\sin(\frac{\pi}{L}x_2)\sin(\frac{\pi}{L}y_2)}{\sinh(\frac{\pi}{2L}(x_1-y_1))^2+\sin(\frac{\pi}{2L}(x_2-y_2))^2}\Big)dxdy - \frac{10L^2}{\lambda}\Bigg)\\
    \geq & \frac{c_0^2}{2}\Bigg( \frac{1}{4\pi }\int_{K}\int_{K}\ln\Big(1+\sin(\frac{\pi}{L}x_2)\sin(\frac{\pi}{L}y_2)\exp(-\frac{\pi}{L}|x_1-y_1|)\Big)dxdy - \frac{10L^2}{\lambda}\Bigg).
\end{align*}
By $\ln(1+t)\geq \frac{1}{2}t$ for $t\in(0,1)$ and $\sin(\frac{\pi}{L}x_2)\sin(\frac{\pi}{L}y_2)\exp(-\frac{\pi}{L}(x_1-y_1))\leq 1$,
\begin{align}
    -I_{\mu,\lambda}& \notag  \\
    % \geq & \frac{c_1^2}{2}\Bigg( \frac{1}{8\pi }\int_{0}^{L}\int_{0}^{L}\int_{0}^{10L}\int_{0}^{10L}\sin(\frac{\pi}{L}x_2)\sin(\frac{\pi}{L}y_2)\exp(-\frac{\pi}{L}|x_1-y_1|)dx_1dy_1dx_2dy_2 - \frac{TL}{\lambda}\Bigg)\\
    \geq & \frac{c_0^2}{2}\Bigg( \frac{1}{8\pi }\Big(\int_{0}^{L}\sin(\frac{\pi}{L}x_2)dx_2\Big)^2\int_{0}^{10L}\int_{0}^{10L}\exp(-\frac{\pi}{L}|x_1-y_1|)dx_1dy_1 - \frac{10L^2}{\lambda}\Bigg) \notag \\
    = & \frac{c_0^2}{2}\Bigg( \frac{1}{8\pi }\Big(\frac{2L}{\pi}\Big)^2 \Big( \frac{2L}{\pi} \cdot \frac{L}{\pi}(\exp(-10\pi)-1)+\frac{20L^2}{\pi}\Big) - \frac{10L^2}{\lambda}\Bigg), \notag \\
    \geq & \frac{c_0^2}{2}\Bigg( \frac{1}{8\pi }\Big(\frac{2L}{\pi}\Big)^3 \Big( -\frac{L}{\pi}+ 10L\Big) - \frac{10L^2}{\lambda}\Bigg) \notag \\
    > & 5c_0^2L^2\Bigg( \frac{3L^2}{\pi^5 } - \frac{1}{\lambda}\Bigg) \notag \\
    = & \frac{5 c_0^2 L^4}{\pi^{5}}=\frac{\mu^2}{5 \pi^5 L^2}>0. \label{ineq:lower bound of -I_mu, lmabda}
\end{align}

% Thus, we have shown that for all $(\mu, \lambda)$ satisfying \eqref{ineq:negativity of I_mu, strip}, $I_{\mu,\lambda} \leq -E_{2,\lambda}[\omega_1] \leq -\frac{\mu^2}{5\pi^5 L^2} < 0$. 
This explicit bound is essential for proving the subsequent properties, such as the propagation speed $W$ (Corollary 4.3) and the strict decrease of $I_{\mu, \lambda}$ (Proposition \ref{prop:strip energy_convergence}). 
With such an estimate, we can now proceed to study more properties of a potential minimizer. A fundamental step is to understand the behavior of the stream function $\psi$, particularly its decay at infinity.

\begin{prop}\label{prop:kinetic energy and stream function, and decay} For $\omega\in L^2 \cap L^1(\S)$ satisfying $x_2\omega\in L^1(\S)$ and $\omega \geq 0 ~ (\omega \not\equiv 0)$, the stream function \eqref{def:stream function} satisfies $\psi>0$ and 
\begin{align}
    \psi \rightarrow 0, \tt{ as } |x_1|\rightarrow \infty, \label{eqn:decay of stream function}
    \end{align}
    as well as
    \begin{align}
    E[\omega]=\frac{1}{2}\|\nabla\psi\|_2^2. \label{eqn:kinetic energy and stream fct}
\end{align}
\end{prop}
\begin{proof} \eqref{eqn:decay of stream function} can be conclude by the observation $0<\psi\leq \psi^H$ in $S$ and $\psi^H\rightarrow 0$ as $|x|\rightarrow \infty$ (see the Proposition 2.2 in \cite{abe2022stability}).

As for the \eqref{eqn:kinetic energy and stream fct}, we take a non-increasing function $\theta \in C_C[0,\infty)$ satisfying $\theta=1$ in $[0,1]$, vanishes in $[2,\infty)$ and set the cut-off function by $\theta_R(x)=\theta(|x|/R)$ in $\S$. Since $-\Delta \psi = \omega $ in $\S$ and $\psi(x_1,0)=\psi(x_1,L)=0$, by multiplying $\psi \theta_R$ by $-\Delta \psi = \omega $ and integration by parts, 
\begin{align*}
    \int_{\S}\psi \omega \theta_R=\int_\S|\nabla\psi|^2\theta_R -\frac{1}{2}\psi^2\Delta \theta_R.
\end{align*}
Since $\psi \rightarrow 0$ as $|x_1|\rightarrow 0$ by \eqref{eqn:decay of stream function}, the second term vanishes as $R\rightarrow \infty$. Hence \eqref{eqn:kinetic energy and stream fct} follows from the monotone convergence theorem.

\end{proof}

As in \cite{abe2022stability}, the positivity of $W$ will be shown to imply compactness of support for minimizers. To demonstrate it, it is necessary to study the regularity of $\psi$. We denote by $\BUC(\S)$ the space of all bounded uniformly continuous functions in $\S$. For an integer $k\geq 0, \BUC^{k+\alpha}(\S)$ denotes the space of all $\psi \in \BUC(\S)$ such that $\p_x^l \psi \in \BUC(\S) \cap C^\alpha(\S)$, for $|l| \leq k$.

\begin{prop}\label{prop:psi/x2 decay}
For any $(\mu,\lambda)$ given in \eqref{ineq:negativity of I_mu, strip}, $\omega \in S_{\mu,\lambda}$, the stream function satisfies $\psi \in \BUC^{2+\alpha}(\overline{\S})$, $0<\alpha<1$, $\psi/x_2 \in \BUC^{1+\alpha}(\S)$, and
\begin{align}
    \frac{\psi(x)}{x_2} \to 0 \quad \text{as } |x_1| \to \infty.\label{eqn:psi/x2 decay}
\end{align}
\end{prop}
\begin{proof}
The proof is standard and follows the strategy of Proposition 2.7 \cite{abe2022stability}.  Since $\omega \in L^1 \cap L^2(\mathcal{S})$, the Biot-Savart law and standard elliptic estimates imply $\nabla^2 \psi \in L^q(\mathcal{S})$ for $q \in (1,2)$ and $\nabla \psi \in L^p(\mathcal{S})$ with $1/p = 1/q - 1/2$.  By the structure equation \eqref{eqn:general structure of minimizer} and the Lipschitz continuity of $f(t)=t_+$, we obtain $\partial_x^l \psi \in L^p_{\mathrm{ul}}(\overline{\mathcal{S}})$ for $|l|=3$. 
The Sobolev embedding then yields $\psi \in \mathrm{BUC}^{2+\alpha}(\overline{\mathcal{S}})$. 
The claim $\psi/x_2 \in \mathrm{BUC}^{1+\alpha}(\overline{\mathcal{S}})$ follows from $\psi(x_1,0)=0$ and the identity
\[
\frac{\psi(x_1,x_2)}{x_2} = \int_0^1 (\partial_2\psi)(x_1, x_2 s)\,ds.
\]
Finally, \eqref{eqn:psi/x2 decay} follows from \eqref{prop:kinetic energy and stream function, and decay} and Hardy's inequality as Proposition 2.7 \cite{abe2022stability}.

\end{proof}

% \begin{proof}
% The regular study is standard, and the proof of this proposition is same as the Proposition 2.7 in \cite{abe2022stability}.

% % Since $\omega \in L^1 \cap L^2$, the representation of stream function in \eqref{eqn:general structure of minimizer} implies $\nabla^2\psi \in L^q$, $q \in (1,2)$, and $\nabla \psi \in L^p$, $1/p = 1/q - 1/2$. By \eqref{eqn:general structure of minimizer} and \eqref{eqn:decay of stream function}, $\psi$ satisfies
% % \begin{equation}\label{eqn:psi system}
% % \begin{aligned}
% %     -\Delta\psi(x) &= \lambda (\psi - W x_2 - \gamma)_+ \quad \text{in } S,\\
% %     \psi &= 0 \quad \text{on } \partial S,\\
% %     \psi &\to 0 \quad \text{as } |x_1| \to \infty.
% % \end{aligned}
% % \end{equation}
% % By the Lipschitz continuity of $\lambda(\cdot)_+$, $\partial_x^l \psi \in L_{\mathrm{ul}}^p(\overline{S})$ for $|l| = 3$. Here $L_{\mathrm{ul}}^p(\overline{S})$ denotes the uniformly local $L^p$-space in $\overline{S}$. Hence $\psi \in \BUC^{2+\alpha}(\overline{S})$ by Sobolev embedding. Since $\psi(x_1,0)=\psi(x_1,L) = 0$ and
% % \begin{align*}
% %     \frac{\psi(x_1,x_2)}{x_2} = \int_0^1 (\partial_2\psi)(x_1, x_2 s)\,ds,
% % \end{align*}
% % $\psi/x_2 \in \BUC^{1+\alpha}(\overline{S})$ follows. By \eqref{eqn:kinetic energy and stream fct} and Hardy's inequality,
% % \begin{align*}
% %     \Bigl\| \frac{\psi}{x_2} \Bigr\|_2 \le 2 \|\nabla \psi\|_2,
% % \end{align*}
% % $\psi/x_2 \in \operatorname{BUC}(\overline{S}) \cap L^2(S)$ and \eqref{eqn:psi/x2 decay} follows.

% \end{proof}

The decay property \eqref{eqn:decay of stream function} is crucial for analyzing the structure equation \eqref{eqn:general structure of minimizer}. Since a minimizer $\omega$ must vanish at infinity, the observation implies the sign of $\gamma$ and then $W$.

\begin{cor}[$\gamma=0$ and $W>0$ for small $\mu$]\label{cor:gamma is zero and positive W} 
    Let $\mu,\lambda$ given by \eqref{ineq:negativity of I_mu, strip}, there exists a constant $M_1 > 0$ such that if $0 < \mu \le M_1$, then every minimizer $\omega \in S_{\mu,\lambda}$ satisfies
    \begin{align}
        \int_{\S}\omega \,dx < 1 \label{ineq:L1 norm of minimizer is strictly smaller that 1}.
    \end{align}
    In particular, $\gamma = 0$ by Remark \ref{remark:measure and gamma for minimizer} and then $W>0$.
\end{cor}
\begin{proof} 
    Recall the structure of $\omega$, \eqref{eqn:general structure of minimizer}.
    % , we shall first analyze the sign of $W$ and $\gamma$. 
    Since $\int_{\S}\omega<\infty$, then $(\psi-Wx_2-\gamma)_+=\frac{1}{\lambda}\omega\rightarrow 0$ as $|x_1|\rightarrow\infty$. Consider a sequence $x_n=(x_{n,1},x_{n,2})$ such that $x_{n,1}\rightarrow \infty$ and $x_{n,2} \rightarrow 0$, then combine with the decay of $\psi$, \eqref{eqn:decay of stream function},
    \begin{align*}
        \limsup_n ( \psi(x_n)-Wx_{2,n}-\gamma )\leq 0,
    \end{align*}
    this implies $\gamma\geq 0$ for all $\mu$. As for $W$, by taking an another sequence $x_n=(x_{n,1},x_{n,2})$ such that $x_{n,1}\rightarrow \infty$ and $x_{n,2} \rightarrow L$, then $-WL-\gamma \leq 0$, so that $W\geq -\gamma/L$.

    To show $\gamma=0$, we noticed that: if $W<0$, then $-W x_2-\gamma \leq -WL-\gamma\leq 0$ implies $\omega=\lambda(\psi-Wx_2-\gamma)_+\leq \lambda \psi$ for any $x_2\in (0,L)$.
    Such inequality also holds when $W\geq 0$. Thus, we can use \eqref{ineq:comparsion between stream function} to obtain
    \begin{align*}
        \int_{0<x_2<2\mu}\omega dx\leq & \lambda \int_{0<x_2<2\mu} \psi dx\leq \lambda \int_{0<x_2<2\mu} \psi^H dx,\\
        \leq & \lambda \int_{0<x_2<2\mu} \int_{\R^2_+} G_H(x,y)\omega(y) dy dx,
    \end{align*}
    and all remaining steps are followed by Remark 2.6 (iii) in \cite{abe2022stability}

    We shall finish the proof by deriving a contradiction. Suppose not, i.e. $W=0$, then $\psi \in L^2(S)$ is a solution to the system:
    \begin{align*}
        \left\{ 
        \begin{aligned}
            -\Delta\psi=\omega=\lambda(\psi)_+&=\lambda \psi && \tt{ in } \S,\\
        \psi&=0 && \tt{ on } \p \S,\\
        \psi &\rightarrow 0 && \tt{ as } |x_1|\rightarrow \infty.
        \end{aligned}
        \right.
    \end{align*}
    Note that $\psi \in \BUC^{2+\alpha}(\overline{\S})$, then it is possible to do a Fourier transform of $\psi$ in $x_1$-variable, we define
    \[
        \widehat{\psi}(\xi,x_2) := \frac{1}{\sqrt{2\pi}} \int_{\mathbb{R}} e^{-i\xi x_1}\,\psi(x_1,x_2)\,dx,
        \qquad \xi\in\R,\; x_2\in(0,L).
    \]
    By Plancherel's theorem, $\widehat\psi\in L^2(\S)$ and
    $\|\psi\|_{L^2(\S)} = \|\widehat\psi\|_{L^2(\S)}$.

    Applying the Fourier transform (in $x_1$) to the equation
    $-\Delta\psi = \lambda\psi$ gives
    \begin{align}\label{eq:transverse}
        \left\{\begin{aligned}
            -\p_{x_2}^2\widehat\psi &= (\lambda-\xi^2)\,\widehat\psi,\\
            \widehat\psi(\xi,0) &= \widehat\psi(\xi,L) = 0\qquad \text{for any }\xi\in\mathbb{R}.
        \end{aligned}\right.
    \end{align}
    For each fixed $\xi$, \eqref{eq:transverse} is a problem for the unknown function $y\mapsto\widehat\psi(\xi,y)$. By solving the elementary single variable boundary valued ODE, $\widehat\psi(\xi,\cdot)$ has a non‑zero solution exists if and only if
    \[
        \lambda - \xi^2 = \nu_m,\qquad 
        \nu_m := \left(\frac{m\pi}{L}\right)^2,\; m=1,2,3,\dots
    \]
    More precisely:
    \begin{itemize}
        \item If $\lambda-\xi^2 \neq \nu_m$ for all $m$, then $\widehat\psi(\xi,\cdot)\equiv 0$ on $[0,L]$.
        \item If $\lambda-\xi^2 = \nu_m$ for some $m$, then $\widehat\psi(\xi,y) = c(\xi)\,\sin\!\bigl(\frac{m\pi y}{L}\bigr)$for some coefficient $c(\xi)$.
    \end{itemize}
    Consequently, the set of $\xi$ for which $\widehat\psi(\xi,\cdot)$ can be non‑zero is contained in
    \[
        E := \bigcup_{m\in\mathbb{N}} \Bigl\{\xi\in\mathbb{R} : \xi^2 = \lambda - \Bigl(\frac{m\pi}{L}\Bigr)^2\Bigr\}.
    \]
    For a fixed $\lambda$, only finitely many $m$ satisfy $\lambda \ge (m\pi/L)^2$; for each such $m$ there are at most two values of $\xi$, namely $\pm\sqrt{\lambda-(m\pi/L)^2}$.  Thus, $E$ is a finite set with Lebesgue measure zero, so that $\|\psi\|_{L^2(\S)} = \|\widehat\psi\|_{L^2(\S)}=0$ and then $\psi\equiv 0$.

    By the structure analysis \eqref{eqn:general structure of minimizer}, $\omega=0$, hence $0>-I_{\mu,\lambda}=E_2[\omega]=0$. This yields a contradiction.

    % $u=0$ is a solution to the system, and by uniqueness, $u=0$ is the unique solution to the system, so that $\psi=0$. By the structure analysis \eqref{eqn:general structure of minimizer}, $\omega=0$, hence $0>-I_{\mu,\lambda}=E_2[\omega]=0$. Contradiction.

\end{proof}

Then combine \eqref{eqn:decay of stream function} with $\gamma=0$ and $W>0$, we have 

\begin{prop}\label{prop:compact support}
Under the assumption in the Proposition \ref{prop:psi/x2 decay}, $\spt\omega$ is compact in $\overline{\mathbb{R}_+^2}$.
\end{prop}

\begin{proof}
Since $\spt\omega = \overline{\{x \in S \mid \psi(x) - W x_2 - \gamma > 0\}}$ for $W > 0$ and $\gamma \ge 0$,
\begin{align*}
    W x_2 \le \psi(x), \qquad x \in \operatorname{spt}\omega.
\end{align*}
Because $\psi/x_2 \to 0$ as $|x_1| \to \infty$ by \eqref{eqn:psi/x2 decay}, the assertion follows.

\end{proof}

\begin{prop}[Steiner symmetrization]\label{prop:steiner symmetrization}
For $\omega \ge 0$ satisfying $\omega \in L^2 \cap L^1(\S)$ and $x_2\omega \in L^1(\S)$, there exists $\omega^* \ge 0$ such that
\begin{align}
    \begin{aligned}
        &\omega^*(x_1,x_2) = \omega^*(-x_1,x_2), \\
        &\omega^*(x_1,x_2) \text{ is non-increasing for } x_1 > 0.
    \end{aligned}
\end{align}
Moreover,
\begin{align}
    \begin{aligned}
        \|\omega^*\|_q &= \|\omega\|_q, \quad 1\le q \le 2, \\
        \|x_2\omega^*\|_1 &= \|x_2\omega\|_1, \\
        E[\omega^*] &\ge E[\omega].
    \end{aligned}
\end{align}
\end{prop}
\begin{proof}
    Recall the structure of the Green's function for $\S$, \eqref{def:green's function for strip}. For any fixed $x_2,y_2$, $G$ is a strictly decreasing function with respect to $|x_1-y_1|$. Then the Riesz Rearrangement inequality can be applied, which directly implies this theorem. See \cite{Fraenkel1974}, Appendix I, and \cite{TurkingtonBruce1983Osvf}, p.1053.

\end{proof}

\begin{prop}\label{prop:far-field estimate}
Let $A,R \ge 1, Q = \{x \in \S \mid |x_1| < AR,\; x_2 < R\}$. Let $\psi$ be the stream function \eqref{def:stream function} for $\omega \in L^2\cap L^1(\S)$ satisfying $x_2\omega \in L^1(\S)$ and $\omega \ge 0$. Assume that \eqref{def:symmetry property} holds for $\omega$. Then,
\begin{align}
    & \psi(x)\leq C x_2^{1/2}\|\omega\|_1^{1/2}\|\omega\|_2^{1/2} + \|\omega\|_1 +x_2\Big(\frac{A}{x_1}\Big)^2\|x_2 \omega\|_1 , x_2 \leq \frac{|x_1|}{A},\label{enq:upperbound estimate}\\
    & \int_{\S\setminus Q} \psi(x)\omega(x)\,dx \leq \frac{C}{\min\{A,R\}^{1/2}} \bigl( \|\omega\|_{L^1\cap L^2}^2 + \|x_2\omega\|_{L^1}^2 \bigr).\label{enq:far-field estimate}
\end{align}
The constant $C$ is independent of $\omega$ and $A, R \ge 1$.
\end{prop}
\begin{proof}
    Combine the observation $0\leq \psi\leq \psi^H$ on $\S$ and the Proposition 3.3, 3.4 in \cite{abe2022stability}, then the result conclude.
    
\end{proof}

\begin{prop}[Convergence of the kinetic energy]
\label{prop:strip energy_convergence}
Assume that the sequence \(\{\omega_n\} \subset L^1  \cap L^2(\S)\) satisfies the following:
\begin{itemize}
    \item \(\sup_n \|\omega_n\|_1 + \sup_n \|\omega_n\|_2 \leq M\) for some constants \(M\), and
    \item \(\omega_n \rightharpoonup \omega\) in \(L^2(\S)\).
\end{itemize}
If each $\omega_n$ is nonnegative and satisfies \eqref{def:symmetry property}, then
\(E[\omega_n] \to E[\omega] \) as \(n \rightarrow \infty\).
\end{prop}
\begin{proof}
    We decompose the energy into two terms
    \begin{align*}
        2E[\omega_n] = \int_{\S} \psi_n(x)\omega_n(x)\,dx = \int_{Q} \psi_n(x)\omega_n(x)\,dx + \int_{\S\setminus Q} \psi_n(x)\omega_n(x)\,dx,
    \end{align*}
    and observe that
    \begin{align*}
        \int_{Q} \psi_n(x)&\omega_n(x)dx \\
        &= \int_{Q} \omega_n(x) \int_{Q} G(x,y)\omega_n(y)\,dydx + \int_{Q} \omega_n(x)\int_{\S\setminus Q} G(x,y)\omega_n(y)dydx,\\
        &= \int_{Q} \omega_n(x) \int_{Q} G(x,y)\omega_n(y)\,dydx + \int_{\S\setminus Q} \omega_n(y)\int_{Q} G(x,y)\omega_n(x)dxdy,\\
        &\leq \int_{Q} \omega_n(x) \int_{Q} G(x,y)\omega_n(y)\,dydx + \int_{\S \setminus Q} \psi_n(x)\omega_n(x)\,dx.
    \end{align*}
    Applying \eqref{enq:far-field estimate} yields
    \begin{align*}
        \Bigl|E[\omega_n] - \frac{1}{2}\int_{Q}\int_{Q} G(x,y)\omega_n(x)\omega_n(y)\,dx\,dy \Bigr| \leq & \int_{\S\setminus Q} \psi_n(x)\omega_n(x)\,dx, \\
        \leq & \frac{C}{\min\{A,R\}^{1/2}}.
    \end{align*}
    By estimating $E[\omega]$ in the same way,
    \begin{align*}
        2|E[\omega_n] - E[\omega]| \le \Bigl| \int_{Q}\int_{Q} G(x,y)(\omega(x)\omega(y) - \omega_n(x)\omega_n(y))\,dx\,dy \Bigr| + \frac{C}{\min\{A,R\}^{1/2}}.
    \end{align*}
    Since $G(x,y) \in L^2(Q \times Q)$ and $\omega_n(x)\omega_n(y) \rightharpoonup \omega(x)\omega(y)$ in $L^2(Q \times Q)$, sending $n\to\infty$ and $A,R\to\infty$ imply the desired result.

\end{proof}

Before stating the next result, we pause to note that the path to Lemma 3.2 in \cite{abe2022stability} has now been fully cleared in our strip setting. Specifically, Proposition \ref{prop:steiner symmetrization} provides the Steiner symmetrization with proper properties, Proposition \ref{prop:strip energy_convergence} guarantees that the weak $L^2$ convergence of the kinetic energy for symmetrized sequences, and Corollary \ref{cor:gamma is zero and positive W} ensures the strict positivity of the propagation speed $W$. With these three components in place, the variational construction in Lemma 3.2 \cite{abe2022stability} carries over verbatim.

\begin{prop}\label{prop:strip exists of minimizer}
For $0 < \nu < \infty, 0< \mu <M_1 \nu, \lambda > \frac{\pi^5}{2L^2}$, there exists a minimizer $\tilde\omega\in \tilde{K}_{\mu,\nu}$ such that
\begin{align}
    E_{2,\lambda}[\tilde\omega]=\inf_{\tilde \omega \in \tilde{K}_{\mu,\nu}} \left\{-E_{2,\lambda}[\tilde \omega]\right\},
\end{align}
where $\tilde{K}_{\mu,\nu} = \Bigl\{ \omega \in L^2 \cap L^1(S) \;\Big|\; \omega \ge 0,\; \int_{S} x_2\omega \,dx \le \mu,\; \int_{S} \omega \,dx \le \nu \Bigr\}.$ Moreover, $\int x_2\omega =\mu$ with a compact support. In particular, it is also a minimizer of \eqref{def:minimizing problem} in ${K}_{\mu,\nu}$.
\end{prop}
\begin{proof}

The result follows by reproducing the proof of Lemma 3.2 in \cite{abe2022stability} in our setting. That argument rests on two pillars: i) a Steiner symmetrization that preserves the constraints while raising the penalized energy (given by Proposition \ref{prop:steiner symmetrization}), and ii) the continuity of the kinetic energy with respect to weak-$L^2$ convergence for a symmetrized sequence (given by Proposition \ref{prop:strip energy_convergence}).

In the final step, one can use the same argument in Lemma 3.2 in \cite{abe2022stability} to prove $\int x_2\omega =\mu$. This implies that such a maximizer in the enlarged space $\tilde K_{\mu,\nu}$ is actually a maximizer in the admissible space $K_{\mu,\nu}$.

% It is a direct consequence of Lemma 3.2 in \cite{abe2022stability} to the strip since we have demonstrated all ingredients.

% since we have demonstrated all ingredients, which are: 
% $$ \exists ~ \omega \in \tilde{K}_{\mu,\nu} \tt{ satisfy \eqref{def:symmetry property} such that } E_{2,\lambda}[\omega] = \sup_{\tilde{\omega} \in \tilde{K}_{\mu,\nu}} E_{2,\lambda}[\tilde{\omega}]  \tt{ and } \int_{\S} x_2\omega \,dx = \mu, $$
% where  $\tilde{K}_{\mu,\nu} = \Bigl\{ \omega \in L^2 \cap L^1(S) \;\Big|\; \omega \ge 0,\; \int_{S} x_2\omega \,dx \le \mu,\; \int_{S} \omega \,dx \le \nu \Bigr\}.$ Thus 
% $$-I_{\mu,\nu,\lambda}\geq E_{2,\lambda}[\omega]=\sup_{\tilde{\omega} \in \tilde{K}_{\mu,\nu}} E_{2,\lambda}[\tilde{\omega}]\geq \sup_{\tilde{\omega} \in {K}_{\mu,\nu}} E_{2,\lambda}[\tilde{\omega}]=-I_{\mu,\nu,\lambda},$$
% we conclude $\omega$ is a minimizer of $I_{\mu,\nu,\lambda}$ in $K_{\mu,\nu}$. $\omega$ has a compact support by the Prop \ref{prop:compact support}.

\end{proof}

\subsection{General Convergence Theorem}

Having established the existence of a minimizer, we now aim for the main result of this section: proving that any minimizing sequence is (up to a subsequence) strongly convergent to a minimizer. The concentration-compactness principle (Proposition \ref{prop:concentration-compactness lemma}) will be the primary tool. However, to apply this principle and, in particular, to exclude the dichotomy scenario, the strict decrease of $I_{\mu, \lambda}$ with respect to $\mu$ is required.
We begin by deriving important uniform bounds for all minimizers with small impulse. These bounds, especially the $L^1$ bound, are indispensable for the perturbation argument that proves the strict decreasing (Proposition \ref{prop:decreasing of I_mu}).

\begin{prop}($L^1$ estimate of minimizers)\label{prop:minimizer L1,2 estimate}
    For any $0<\nu<\infty, 0<\mu < M_1 \nu $ and $\lambda \geq \pi^5(2 L^2)^{-1}$, we have
    \begin{align}
        \sup_{\omega\in S_{\mu,\nu,\lambda }} |\spt\omega|\leq \frac{5}{2}\pi^5 \nu^{-2} L^2. \label{ineq:minimizer support estimate}
    \end{align}
    Consequently,
    \begin{align}
        \sup_{\omega\in S_{\mu,\nu,\lambda }}\| \omega \|_1\leq C^{3} \mu\nu^{-3}\lambda^{2} L^3, \label{ineq:minimizer L1 estimate}\\
        \sup_{\omega\in S_{\mu,\nu,\lambda }}\| \omega \|_2\leq C^{3} \mu\nu^{-2}\lambda^{2} L^2, \label{ineq:minimizer L2 estimate}
    \end{align}
    for some constant $C$ independent from $\mu,\nu,\lambda,L$.
\end{prop}
\begin{proof}
    It is sufficient to consider the case when $\nu=1$. For any $\omega\in S_{\mu,\lambda}$, with $\mu\leq M_1$ and $\lambda \geq \pi^5(2 L^2)^{-1}$. Recall the equation \eqref{eqn:general structure of minimizer} with $W=W(\omega)>0 $ and $\gamma=0$. Then, multiplying the equation $-\Delta(\psi-W x_2)=\lambda(\psi-W x_2)_+$ by $\psi-Wx_2$ and integrating by parts,
    \begin{align*}
        \int_{\spt\omega} |\nabla \psi |^2-2W \partial_2\psi +W^2dx = & \int_{\spt\omega} |\nabla (\psi-Wx_2)|^2dx,\\
        = &\lambda \int_{\spt\omega}|\psi -Wx_2|^2dx=\frac{1}{\lambda} \int_{\spt\omega}|\omega|^2dx.
    \end{align*}
    
    Since $\psi-W x_2=0$ on $\p ( \spt\omega ) $, we have
    \begin{align*}
        \int_{\spt\omega} \partial_2\psi = & \int_{\spt\omega} \partial_2(\psi-Wx_2) + W|\spt\omega|, \\ 
        = & \int_{\p ( \spt\omega )} n_2(\psi-Wx_2) + W|\spt\omega|=W|\spt\omega|.
    \end{align*}
    Then 
    \begin{align*}
        2E[\omega] - W^2|\spt\omega| & \geq \int_{\spt\omega} |\nabla \psi |^2-W^2|\spt \omega|= \frac{1}{\lambda} \int_{\spt\omega}|\omega|^2dx.
    \end{align*}
    Hence,
    \begin{align*}
        W^2|\spt\omega| & \leq 2E_{2,\lambda }[\omega]\leq -2I_{\mu,\lambda}.
    \end{align*}
    Meanwhile, 
    \begin{align}
        \begin{aligned}
            -I_{\mu,\lambda} = & E_{2,\lambda}[\omega]=\frac{1}{2}\int_{\spt\omega} \psi \omega -\frac{1}{2\lambda}\int_{\spt\omega} \omega^2, \\
            = & \frac{1}{2}\int_{\spt\omega} (\psi-Wx_2) \omega +\frac{1}{2}W\mu -\frac{1}{2\lambda}\int_{\spt\omega} \omega^2=\frac{1}{2}W\mu,
        \end{aligned}
         \label{eqn:value of I_mu and W}
    \end{align}
    then 
    \begin{align*}
        |\spt\omega| & \leq \mu/W.
    \end{align*}
    Recall there exists a lower bound for $-I_{\mu,\lambda} \geq \frac{\mu^2}{5 \pi^5 L^2}$, then the identity \eqref{eqn:value of I_mu and W} implies 
    \begin{align*}
        W\geq \frac{2 \mu}{5 \pi^5 L^2}.
    \end{align*}
    Hence 
    \begin{align*}
       |\spt\omega| & \leq \mu/W \leq \frac{5}{2}\pi^5 L^2,
    \end{align*}
    and the proof finished by the Cauchy inequality and the $L^2$ uniform bound \eqref{ineq:uniformly bounded} while $\omega_n\equiv \omega$ for all $n$:
    \begin{align*}
        \|\omega\|_1\leq & \|\omega\|_{2}|\spt\omega|^{1/2}\leq C\mu^{1/3}\lambda^{2/3}L\|\omega\|_1^{2/3},\\
        \Rightarrow \|\omega\|_1 \leq & C^3\mu \lambda^{2}L^3,\\
        \Rightarrow \|\omega\|_2 \leq & C\mu^{1/3}\lambda^{2/3}\|\omega\|_1^{2/3}\leq C^3\mu\lambda^{2}L^2.
    \end{align*}
    
\end{proof}

\begin{remark}\label{remark:uniqueness of W}
    The proof shows that, for a small $\mu$, the parameter $W$ is independent of the choice of $\omega$, i.e., $W=-2I_{\mu,\lambda}/\mu$.
\end{remark}

\begin{remark}\label{remark:order of I_mu}
 Combined with estimates \eqref{ineq:estimate of kinetic energy}, \eqref{ineq:minimizer L1 estimate}, and \eqref{ineq:minimizer L2 estimate}, together with the lower bound \eqref{ineq:lower bound of -I_mu, lmabda}, one can confirm $-I_{\mu,\lambda}\in\mathcal{O}(\mu^2)$ as $\mu$ is sufficiently small.
\end{remark}

Before proving the strict monotonicity, we recall that in Proposition~\ref{prop:propoties of I_mu} and the analysis for \eqref{ineq:negativity of I_mu, strip} we constructed an explicit trial function $\omega_0 = c_0 \mathbf{1}_{(0,10L)\times(0,L)}$ supported on a fixed compact set $K\subset\mathcal{S}$ with $E_{2,\lambda}[\omega_0] > 0$. This function serves as a “seed” for perturbing minimizers without violating the mass constraint, thanks to the uniform $L^1$ bound provided by Proposition~\ref{prop:minimizer L1,2 estimate}.

\begin{prop}(Strictly decreasing of $I_\mu$)\label{prop:decreasing of I_mu}
    For any $0<\nu<\infty$ and $\lambda \geq \pi^5(2 L^2)^{-1},$ take $M_2=\min\{C^{-3}\nu^3 \lambda^{-2}L^{-3},M_1\nu\}>0$, then for any $0<\alpha<\mu<M_2$, $I_{\alpha,\nu,\lambda} > I_{\mu,\nu,\lambda}$.
\end{prop}
\begin{proof}
    It is sufficient to consider the case $\nu=1$. For any $0<\mu <M_2$, fix a $\omega_\mu\in S_{\mu,\lambda}$. Note that $\omega_\mu$ is compactly supported has an $L^1$ upper bound $C^3 \mu\lambda^{2} L^3$ 
    and recall the function $\omega_0$ given in Proposition \ref{prop:propoties of I_mu}. For $s\in(0,1-C^3 \mu\lambda^{2} L^3)$, we define $\omega_{s}\in K_{\mu(1+s)}$ by 
    \begin{align*}
        \omega_s(x)=\omega_{\mu}(x)+s\omega_0(x+\tau e_1), 
    \end{align*}
    where $\tau $ is sufficiently large constant such that $ \spt(\omega_\mu) \cap \spt(\omega_0(\cdot+\tau e_1))=\emptyset$. Then for any $\alpha\in (\mu,~\mu(1+1-C^3 \mu\lambda^{2} L^3) ),$ 
    \begin{align*}
        -I_{\alpha, \lambda} & \geq E_{2,\lambda}[\omega_s]=E_{2,\lambda}[\omega_\mu]+s^2 E_{2,\lambda}[\omega_0]+2s\int_S\int_S G(x,y)\omega_\mu(x)\omega(y+se_1)dxdy,\\
        & > -I_{\mu,\lambda}+\frac{1}{5 \pi^5 L^2}\mu^2s^2,\\
        & =-I_{\mu,\lambda}+\frac{1}{5 \pi^5 L^2}(\alpha-\mu)^2.
    \end{align*}
    Here, $s=\alpha/\mu -1$ and $\alpha=\mu(1+s)$. Then $\mu$ induce a monotone increasing and bounded above sequence $\{ \mu_k \}_{k}$ by $\mu_0=\mu, \mu_{k+1}=\mu_k(1+1-C^3 \mu_k\lambda^{2} L^3)$ with a limit $C^{-3}\lambda^{-2}L^{-3}$. Hence, for any $\mu<\alpha<M_2$, there exists some $k\geq 0$ such that $\alpha\in(\mu_k,\mu_{k+1}]$, then $I_{\alpha,\lambda}<I_{\mu_k,\lambda}\leq I_{\mu,\lambda}$.
    
\end{proof}

\begin{prop}\label{prop:concentration-compactness lemma}
Let \(0<\mu<\infty\). Let \(\{\rho_n\}\subset L^1(\S)\) satisfy
\begin{align*}
\rho_n\ge 0,\quad n\ge 1,\qquad \int_{S}\rho_n\,dx = \mu_n \to \mu \quad\text{as } n\to\infty .
\end{align*}
There exists a subsequence \(\{\rho_{n_k}\}\) satisfying one of the following:
\begin{enumerate}
\item \textbf{(Compactness)} There exists a sequence \(\{y_k\}\subseteq\overline{\S}\) such that \(\rho_{n_k}(\cdot+y_k)\) is tight, i.e., for every \(\epsilon>0\) there exists \(R>0\) such that
\begin{align*}
\liminf_{k\to\infty}\int_{B(y_k,R)\cap \S}\rho_{n_k}\,dx \ge \mu-\epsilon.
\end{align*}
In addition, it is possible to assume $y_{k,2}=0$ for all $k$.

\item \textbf{(Vanishing)} For each \(R>0\)
\begin{align*}
\lim_{k\to\infty}\sup_{y\in \S}\int_{B(y,R)\cap \S}\rho_{n_k}\,dx = 0 .
\end{align*}

\item \textbf{(Dichotomy)} There exists \(\alpha\in(0,\mu)\) such that for every \(\epsilon>0\) there exist \(k_0\ge 1\) and \(\{\rho_k^1\},\{\rho_k^2\}\subset L^1(\S)\) such that \(\spt\rho_k^1\cap\spt\rho_k^2=\emptyset\), \(0\le\rho_k^i\le\rho_{n_k}\) for \(i=1,2\), and
\begin{align*}
\limsup_{k\to\infty}\Bigl\{\|\rho_{n_k}-\rho_k^1-\rho_k^2\|_{L^1}
+ \Bigl|\int_{S}\rho_k^1\,dx-\alpha\Bigr|
+ \Bigl|\int_{S}\rho_k^2\,dx-(\mu-\alpha)\Bigr|\Bigr\} \le \epsilon ,
\end{align*}
\begin{align*}
d(\spt\rho_k^1,\spt\rho_k^2) \to \infty \quad\text{as } k\to\infty .
\end{align*}
\end{enumerate}
\end{prop}

\begin{proof}
The assertion is proved by Lemma I.1 in \cite{LIONS1984109} for $\R^2$ and the fixed mass \(\mu_n=\mu\) by using Lévy's concentration function. The case of $\mu_n \rightarrow \mu$ is proved by the Lemma 4.1 in \cite{abe2022stability} for $\R_2^+$. This proof is a direct consequence of result in \cite{abe2022stability} by view $L^i(\S)=\{\omega\in L^i(\R_+^2) ~ | ~ \spt\omega\subseteq \overline{\S} \}$. From this perspective, the sequence of $\rho_n$ meets the requirement of Lemma 4.1 in \cite{abe2022stability} and then yields three possibilities: compactness, vanishing, or dichotomy.

% Consider the sequence $\{\rho_n\} \subset L^1(\mathcal{S})$. Extending each $\rho_n$ by zero 
% outside $\mathcal{S}$ yields a sequence of finite measures on $\mathbb{R}^2$ with total mass 
% $\mu_n \to \mu$. The concentration-compactness lemma of Lions 
% \cite[Lemma I.1]{Lions1984} applies directly in $\mathbb{R}^2$, yielding three 
% possibilities: compactness, vanishing, or dichotomy.

% The statements for vanishing and dichotomy translate verbatim because the $L^1$ norm on 
% $\mathcal{S}$ coincides with that on $\mathbb{R}^2$ after zero extension, and the 
% disjointness of supports is unaffected by the domain restriction.

\begin{enumerate}
\item \textbf{(Compactness)} In this case, the tightness is originally stated with centres $y_k \in \mathbb{R}^2_+$. If $\limsup_k y_{k,2}=\infty$, then for large $k$, the ball $B(y_k, R)$ would not intersect $\mathcal{S}$, forcing the integral to vanish and contradicting tightness. Therefore, $\sup_k y_{k,2} < \infty$. By replacing original $R$ by $R+\limsup_k y_{k,2}$, it is possible to assume $y_{2,k}=0$ for all $k$, so that $y_k \in \overline{\S}$ for all $k$.

\item \textbf{(Vanishing)} In this case, one can use the set relation $\S\in \R^2_+$ to conclude for any \(R>0\),
\begin{align*}
0\leq \lim_{k\to\infty}\sup_{y\in \S}\int_{B(y,R)\cap \S}\rho_{n_k}\,dx \leq \lim_{k\to\infty}\sup_{y\in \R_+^2}\int_{B(y,R)\cap \R_+^2}\rho_{n_k}\,dx = 0 .
\end{align*}

\item \textbf{(Dichotomy)} In this case, due to \(0\le\rho_k^i\le\rho_{n_k}\), then $\spt \rho_{n_k}^i\subseteq \spt\rho_k \subseteq \overline{\S}$ for \(i=1,2\), so that \(\{\rho_k^1\},\{\rho_k^2\}\subset L^1(\S)\) and conclude the result.
\end{enumerate}

\end{proof}

Under the strict decreasing of \(I_{\mu,\lambda}\) (Proposition \ref{prop:decreasing of I_mu}), the convergence of kinetic energy for Steiner-symmetrized sequences (Proposition \ref{prop:strip energy_convergence}), and the concentration-compactness alternatives (Proposition \ref{prop:concentration-compactness lemma}), we can now outline the proof of the general convergence theorem for the strip. This theorem parallels the strategy of Abe \& Choi \cite{abe2022stability} for the half-plane, but its validity in the strip domain \(\mathcal{S}\) hinges on the adaptations made in the preceding propositions.

The proof proceeds by eliminating vanishing and dichotomy for a minimizing sequence \(\{\omega_n\}\). Vanishing is ruled out by the negativity of \(I_{\mu,\lambda}\). The more delicate case is dichotomy, where a minimizing sequence splits into two well-separated components. By applying Steiner symmetrization to these components and using the strict monotonicity of \(I_{\mu,\lambda}\), we deduce a contradiction to its minimality, forcing compactness instead. We now formalize this result.

% We have completed all the materials required for the general convergence theorem in \cite{abe2022stability} in the Class \ref{def:case Omega=S}. Therefore, the general convergence theorem for the strip model can be obtained through the same step; we omit the proof.

\begin{thm}\label{thm:general convergence theorem for strip}
    Let $0<\nu<\infty, \mu<M_2 \nu$ and $\lambda \geq \pi^5(2 L^2)^{-1}$. For any minimizing sequence $\{\omega_n\}$ satisfying $\omega_n\in K_{\mu_n,\nu}$, $\mu_n \rightarrow \mu $ and $-E_{2,\lambda}[\omega_n] \rightarrow I_{\mu,\nu,\lambda}$, then there exists a subsequence, still denoted by $\{\omega_n\}$, s.t. there exists $\omega\in K_{\mu,\nu}$, 
    $\omega_n\rightarrow\omega$ and $x_2\omega_n\rightarrow x_2\omega$ strongly in $L^2(\S)$ and $L^1(\S)$ respectively.
\end{thm}
\begin{proof}
The proof follows the concentration-compactness argument of Theorem 1.3 in \cite{abe2022stability}, which we briefly outline. By Proposition \ref{prop:concentration-compactness lemma}, the sequence $x_{2}\omega_{n}$ (up to a subsequence) satisfies one of three scenarios: compactness, vanishing, or dichotomy. Vanishing cannot occur since $I_{\mu,\lambda} < 0$. If dichotomy were to occur, the sequence would split into two non-trivial parts with strictly smaller impulses. Applying Steiner symmetrization (Proposition \ref{prop:strip energy_convergence}) and the strict monotonicity of $I_{\mu,\lambda}$ (Proposition \ref{prop:decreasing of I_mu}), one derives a contradiction to the minimizing property of $\{\omega_{n}\}$, as shown in detail in the Theorem~1.3, \cite{abe2022stability}. Therefore, the sequence must be locally compact up to horizontal translations along the boundary $\S$. The strong convergence of the kinetic energy then follows from Proposition \ref{prop:strip energy_convergence}, which identifies the limit as a minimizer.
\end{proof}

\section{Stability respect to the minimizer}\label{sec:stability}

Having established the existence of minimizers and a general convergence theorem, we now turn to their orbital stability with $C_c^{\infty}(\Omega)$ initial data. Using the strong convergence shown in Theorem \ref{thm:exists of minimizer} for $\Omega=\D$ and Theorem \ref{thm:general convergence theorem for strip} for $\Omega=\S$, we can show that the set of minimizers is orbital stable in the sense that any flow starting close to a minimizer will remain close to the set of minimizers during its lifespan. We conclude this section by deriving the precise functional form of these minimizers.

The proof here is the same as that for Theorem 1.4 in \cite{abe2022stability}, but to ensure completeness, we present it below.

\begin{thm}\label{thm:stable wrt minimizer}
    Let $\nu,\mu,\lambda$ given in the Theorem \ref{thm:exists of minimizer} for the case $\Omega=\D$ and given in the Theorem \ref{thm:general convergence theorem for strip} for the case $\Omega=\S$, then $S_{\mu,\nu,\lambda}$ is orbitally stable in the sense that: for any $\epsilon>0$, there exists $\delta>0$ s.t. $\forall ~ \xi_0\in C_c^{\infty}(\Omega), ~ x_2\xi_0\in L^1(\Omega), ~ \xi_0\geq0, ~ \|\xi_0\|\leq \nu$ and 
    \begin{align*}
        \inf_{\omega\in S_{\mu,\nu,\lambda}}\Big\{ \|\xi_0-\omega\|_{2}+\|x_2(\xi_0-\omega)\|_{1} \Big\}\leq \delta,
    \end{align*}
    the solution $\xi(t)$ satisfy: $\forall ~ t \in [0,T_{\text{lifespan}})$,
    \begin{align}
        \inf_{\omega\in S_{\mu,\nu,\lambda}}\Big\{ \|\xi(t)-\omega\|_{2}+\|x_2(\xi(t)-\omega)\|_{1} \Big\}\leq \epsilon. \label{eqn:stable result}
    \end{align}
\end{thm}

\noindent\textit{Proof: } It suffices to prove the case when $\nu=1$. Suppose that \eqref{eqn:stable result} is false. Then there exists $\epsilon_0>0$, a sequence $\{\xi_{0,n}\}\subseteq C_c^\infty(\mathcal{D}), ~ \xi_{0,n} \geq 0, ~ \|\xi_{0,n}\|_1\leq 1 $ and $t_n$ in the lifespan of the solution $\xi_n$ based on the initial data $\xi_{0,n}$ s.t. 
\begin{align*}
    \inf_{\omega\in S_{\mu,\lambda}}\Big\{ \|\xi_{0,n}-\omega\|_{2}+\|x_2(\xi_{0,n}-\omega)\|_{1} \Big\} & \leq 1/n, \\ \inf_{\omega\in S_{\mu,\lambda}}\Big\{ \|\xi_n(t_n)-\omega\|_{2}+\|x_2(\xi_n(t_n)-\omega)\|_{1} \Big\} &> \epsilon_0.
\end{align*}
We write $\xi_n=\xi_n(t_n)$ by suppressing $t_n$. We take $\omega_n\in  S_{\mu,\lambda}$ such that $\|\xi_{0,n}-\omega\|_{2}+\|x_2(\xi_{0,n}-\omega)\|_{1}\rightarrow 0$. By \eqref{ineq:estimate of energy difference}, 
$$ | E_{2,\lambda}[\xi_{0,n}]+I_{\mu,\lambda} | = | E_{2,\lambda}[\xi_{0,n}]-E_{2,\lambda}[\omega_n] | \rightarrow 0 \text{ as } n \rightarrow \infty$$
Thus $\{\xi_{0,n}\}$ be a minimizing sequence s.t. $\xi_{0,n}\in K_{\mu_n}, \mu_n\rightarrow\mu$ and $-E_{2,\lambda}[\xi_{0,n}] \rightarrow I_{\mu,\lambda} $ as $n\rightarrow\infty$.

Knowing that $\xi$ conserved on particle trajectory map, then $\|\xi_n\|_{2}=\|\xi_{0,n}\|_{2}$ and $\|x_2\xi_n\|_1=\|x_2\xi_{0,n}\|_1$. Meanwhile, since the kinetic energy is conserved in time, then $E_{2,\lambda}[\xi_{0,n}]=E_{2,\lambda}[\xi_{n}]$ for any $n$.

Hence $\{\xi_{n}\}$ be a minimizing sequence s.t. $\xi_{n}\in K_{\mu_n}, \mu_n\rightarrow\mu$ and $-E_{2,\lambda}[\xi_{n}] \rightarrow I_{\mu,\lambda} $ as $n\rightarrow\infty$. By the Theorem \ref{thm:exists of minimizer}, by choosing a subsequence still denoted by $\{\xi_{n}\}$, there exists $\xi\in S_{\mu,\lambda}$ s.t. $\omega_n\rightarrow\omega$ and $x_2\omega_n\rightarrow x_2\omega$ in $L^2(\Omega)$ and $L^1(\Omega)$ respectively. Sending $n\rightarrow\infty$, 
\begin{align*}
    0=&\inf_{\omega\in S_{\mu,\lambda}}\Big\{ \|\xi-\omega\|_{2}+\|x_2(\xi-\omega)\|_{1} \Big\},\\
    =&\inf_{\omega\in S_{\mu,\lambda}}\Big\{ \lim_{n\rightarrow\infty}\Big( \|\xi_{n}-\omega\|_{2}+\|x_2(\xi_{n}-\omega)\|_{1}\Big) \Big\},\\
    \geq&\liminf_{n\rightarrow\infty}\Big(\inf_{\omega\in S_{\mu,\lambda}}\Big\{ \|\xi_{n}-\omega\|_{2}+\|x_2(\xi_{n}-\omega)\|_{1}\Big\} \Big),\\
    \geq&\epsilon_0.
\end{align*}
We obtained a contradiction.

\rightline{$\Box$}

\section{Discussion}\label{sec:discussion}

In the end, we discuss some generalizations of Theorem \ref{thm:stable wrt minimizer}, and state some open questions.

\subsection{More general domains.}
A natural question arises as to whether the stability result applies to more general domains. It would be valuable to identify optimal conditions on the domain that ensure both the weak continuity of kinetic energy and the existence of minimizers. This is particularly relevant for domains with even weaker volume conditions and small perturbations of the strip model. Examples of such domains include shapes defined by functions like \((\ln(2 + |x_1|))^{-1}\), \(1/x\), and \(e^{-x^2/2} + 1\), along with \(\partial \mathbb{R}^2_+\) as boundaries, among others. Their behavior is unpredictable due to the absence of technical tools, such as rearrangement, and the concentrated nature resulting from the decay of the domain itself.

\subsection{Free-boundary problems.}
Since the Class \ref{def:case Omega=D} is very general, for a free-boundary problem whose initial domain satisfies the decay assumption \eqref{def:weak_finite_volume_domain_assumption}, it is believed that within a certain time, the free-boundary domain $\D_t$ will still satisfy the condition \ref{def:weak_finite_volume_domain_assumption}. This holds particularly for the domains studied in Hu--Luo--Yao \cite{hu2024small}. 
A pertinent question arises regarding whether the stability issues discussed in this paper can be extended to the free-boundary case. This concept may be further explored by referencing the approximate Biot-Savart law proposed in \cite{hu2024small}. However, it is important to note that their study focuses on local properties, whereas stability studies necessitate a reliance on global considerations, which requires more advanced techniques.
%Therefore, a natural question is whether the stability issue discussed in this paper can be generalized to the free-boundary case. The general idea might be developed by referring to the approximate Biot-Savart law proposed in \cite{hu2024small}. However, the difference lies in that they study local properties, while stability studies need to rely on the global for discussion, which requires more techniques.

\subsection{Multiple patches and more study of minimizers.}

Our results extend the work of \cite{abe2022stability} from a single vortex patch to more general domains. A natural next step is to consider configurations with multiple patches, as studied for dipole pairs in \cite{choi2024stability}. However, extending our approach to dipole pairs presents two major difficulties.

First, the structure of the minimizer set $S_{\mu,\nu,\lambda}$ remains largely unknown for both domain classes $\D$ and $\S$. Consequently, a vortex evolving under the Euler dynamics might approach minimizers with different spatial structures over its lifespan, making stability analysis significantly more challenging.

Second, the method of \cite{choi2024stability} introduces an $L^2$-norm equality constraint, which allows one to reinterpret minimizers of the penalized energy as maximizers of the kinetic energy. Their method heavily relies on scaling invariance, a property that is absent in our domains, as discussed in Section \ref{sec: difficulty_loss_scaling}. Hence, a completely new approach is required to handle the multiple patch case, which is considerably more challenging.

Also, in \cite{abe2022stability}, the author figured out the formula of the minimizer by the moving plane method in Theorem 4.2 of \cite{fraenkel2000introduction}. Unfortunately, when studying similarly, our model can only achieve a general structure, as outlined in Theorem \ref{prop:general structure}. However, due to the lack of general symmetry and the specific characteristics of the domain, such as \(\R^2_+\), we are unable to derive more detailed properties. It may be possible to solve for additional useful properties by fixing a specific domain. For instance, we could explore a minimizer with compact supports (for weak finite volume domains), connected supports, boundedness, symmetry, and even an explicit formula. If we can achieve these results within this domain, we believe it may be possible to eliminate the \(L^1\) upper bound in Theorem \ref{thm:stable wrt minimizer}, similar to what was demonstrated in \cite{Abe2025StabilityOL}.
%Unfortunately, when studying similarly, our model can only obtain a general structure like the Theorem \ref{prop:general structure}, but because the domain itself loses the general symmetry and particularity like $\R^2_+$, it cannot obtain more properties. Perhaps after fixing a specific domain, it is possible to solve for more useful properties, such as a minimizer having compact supports (for weak finite volume domain), connected supports, boundedness, symmetry, and even an explicit formula. If these can be achieved in this domain, we believe that one can get rid of the $L^1$ upper bound in Theorem \ref{thm:stable wrt minimizer} as in \cite{Abe2025StabilityOL}.

\bibliographystyle{plain}
\bibliography{sample}

\end{document}